\documentclass[12pt]{article}

\usepackage{amssymb}
\usepackage{amsthm}
\usepackage{amsmath}
\usepackage{graphicx}
\usepackage{fullpage}
\usepackage{color}
\usepackage{enumerate}
 \numberwithin{equation}{section}
\usepackage{booktabs}
\allowdisplaybreaks

\usepackage[pdftex]{hyperref}
 \usepackage{mathtools}
 \mathtoolsset{showonlyrefs}

\theoremstyle{plain}
\newtheorem{thm}{Theorem}[section]
\newtheorem{cor}[thm]{Corollary}

\newtheorem{lem}[thm]{Lemma}
\newtheorem{prop}[thm]{Proposition}

\newtheorem*{fcl}{Finite Chain Lemma}

\theoremstyle{definition}

\theoremstyle{remark}

\newtheorem{rem}[thm]{Remark}

\newcommand{\N}{\mathbb{N}}
\newcommand{\R}{\mathbb{R}}

\newcommand{\bp}{\begin{proof}[\ensuremath{\mathbf{Proof}}]}
\newcommand{\bs}{\begin{proof}[\ensuremath{\mathbf{Solution}}]}
\newcommand{\ep}{\end{proof}}
\newcommand{\be}{\begin{equation}}
\newcommand{\ee}{\end{equation}}

\begin{document}

% Title
\title{Extremal functions for Morrey's inequality}

% Name
\author{Ryan Hynd\footnote{Department of Mathematics, University of Pennsylvania.  Partially supported by NSF grant DMS-1554130.}\;  and Francis Seuffert\footnote{Department of Mathematics, University of Pennsylvania.}}

\maketitle 

%  Abstract  
\begin{abstract}
We give a qualitative description of extremals for Morrey's inequality. Our theory is based on exploiting the invariances of this inequality, studying the equation satisfied by extremals and the observation that extremals are optimal for a related convex minimization problem. 
\end{abstract}

%\tableofcontents

%%%%%%%%%%%%%%%%%%%%%%%%%%%%%%%%%%%%%%%%%%%%%%%%%%%%%%%%%%%%%%%%%%%%%%%%%%%
\section{Introduction}
% Sobolev 
The Sobolev inequality is arguably the most important functional inequality in the theory of Sobolev spaces. It asserts that for each natural number $n$ and 
$$
1< p<n,
$$
there is a constant $C$ such that  
\be\label{GNS}
\left(\int_{\R^n}|u|^{p^*}dx\right)^{1/p^*}\le 
C\left(\int_{\R^n}|Du|^{p}dx\right)^{1/p}
\ee
for each continuously differentiable $u:\R^n\rightarrow \R$ with compact support.   Here
\be\label{ThepStar}
p^*=\frac{np}{n-p}.
\ee
This estimate is ubiquitous in PDE theory and also played a crucial role in the solution of Yamabe's problem in Riemannian geometry \cite{MR888880}. For Yamabe's problem, 
it was essential to know the {\it sharp} or smallest constant $C=C^*$ for which \eqref{GNS} holds and to know the {\it extremals} or functions for which equality is attained in \eqref{GNS}.

\par Using symmetrization methods, Talenti derived the sharp constant 
$$
C^*=\frac{1}{\sqrt{\pi}\;n^{1/p}}\left(\frac{p-1}{n-p}\right)^{1-1/p}\left(\frac{\Gamma\left(1+\frac{n}{2}\right)\Gamma(n)}{\Gamma\left(\frac{n}{p}\right)\Gamma\left(1+n-\frac{n}{p}\right)}\right)^{1/n}
$$
and found that nonnegative extremals are necessarily of the form
$$
u(x)=\left(a+b|x-z|^{\frac{p}{p-1}}\right)^{1-\frac{n}{p}}
$$
for some $a,b>0$ and $z\in \R^n$ \cite{MR0463908}.  Aubin made similar insights around the same time in his influential work on Yamabe's problem \cite{MR0431287, MR0448404}. In particular, both Talenti and Aubin noted that extremals for the Sobolev inequality are radially symmetric.

% Morrey
%\par Perhaps the second most important estimate in the theory of Sobolev spaces is Morrey's inequality. It states that for each natural number $n$ and 
%\be\label{pBiggerthann}
%n<p<\infty,
%\ee
%there is a constant $C$ for which the inequality
%\be\label{MorreyIneq}
%\sup_{x\ne y}\left\{\frac{|u(x)-u(y)|}{|x-y|^{1-n/p}}\right\}\le C \left(\int_{\R^n}|Du|^pdz\right)^{1/p}
%\ee
\par The counterpart for Sobolev's inequality when 
\be\label{pBiggerthann}
n<p<\infty
\ee
is known as Morrey's inequality. It asserts that there is a constant $C$ for which the inequality
\be\label{MorreyIneq}
\sup_{x\ne y}\left\{\frac{|u(x)-u(y)|}{|x-y|^{1-n/p}}\right\}\le C \left(\int_{\R^n}|Du|^pdz\right)^{1/p}
\ee
holds for every $u:\R^n\rightarrow \R$ that is continuously differentiable.  It is of great interest to deduce the sharp constant $C=C_*$ and to explicitly express the corresponding extremal functions.  These problems are easy to 
solve when $n=1$ but are unsolved for $n\ge 2$.  In particular, symmetrization methods do not seem to yield insightful information for 
Morrey's inequality the way they do for other well known functional inequalities as detailed in \cite{ MR2178968, MR2395175,MR2551138, MR2402947, MR1322313}.

% What we do in this paper 
\par In this paper, we show extremals for Morrey's inequality exist and establish several qualitative properties of these functions.  Specifically, we study 
their symmetry properties and show that they are unique up to the natural invariances associated with Morrey's inequality.  We also verify Morrey extremals 
are bounded, smooth on $\R^n$ minus two points, and their level sets bound convex regions in $\R^n$. Our approach is based on the analysis of a certain PDE that extremals satisfy.  We will informally summarize these properties below and then verify precise versions of these statements in the discussion to follow.  Hereafter, the term ``extremal" will only apply to Morrey's inequality and we will always assume \eqref{pBiggerthann}.
\\
% Properties
\par {\bf  Maximized H\"older ratio}.  For each extremal function $u$, there are distinct points $x_0,y_0\in \R^n$ that maximize its $1-n/p$ H\"older ratio. That is,
\be\label{MaxHoldeRatio}
\sup_{x\ne y}\left\{\frac{|u(x)-u(y)|}{|x-y|^{1-n/p}}\right\}=
\frac{|u(x_0)-u(y_0)|}{|x_0-y_0|^{1-n/p}}.
\ee
This property also holds more generally for functions  $v:\R^n\rightarrow\R$ which satisfy 
\be\label{DVLp}
\int_{\R^n}|Dv|^pdx<\infty.
\ee
Curiously, a stability estimate for Morrey's inequality follows from this insight. In particular, for any $v$ satisfying \eqref{DVLp}, there is an extremal $u$ such that
\be
\left(\frac{C_*}{2}\right)^p\int_{\R^n}|Du-Dv|^pdx+\sup_{x\ne y}\left\{\frac{|v(x)-v(y)|}{|x-y|^{1-n/p}}\right\}^p\le C_*^p\int_{\R^n}|Dv|^pdx
\ee
when $2<p<\infty$; an analogous inequality holds for $1<p\le 2$. 
\\
\par {\bf PDE}. Suppose $u$ is an extremal which satisfies \eqref{MaxHoldeRatio}. Then the PDE
\be\label{strongPDE}
-\Delta_pu=\frac{|u(x_0)-u(y_0)|^{p-2}(u(x_0)-u(y_0))}{C_*^p|x_0-y_0|^{p-n}}(\delta_{x_0}-\delta_{y_0})
\ee
holds in $\R^n$.  Here $\Delta_p$ is the $p$-Laplace operator 
$$
\Delta_pu=\text{div}(|Du|^{p-2}Du),
$$
and $C_*$ is the sharp constant for Morrey's inequality.  Moreover, the converse is true. If $v$ satisfies
\be
-\Delta_pv=c(\delta_{x_0}-\delta_{y_0})
\ee
in $\R^n$ for some $c\in \R$, then $v$ is necessarily an extremal and the H\"older ratio of $v$ is maximized at $x_0$ and $y_0$.
\\
\par {\bf Uniqueness}. For each distinct pair of points $x_0,y_0\in \R^n$ and distinct pair of values $\alpha,\beta\in \R$, there is a unique extremal $u$ which satisfies \eqref{MaxHoldeRatio} and
$$
u(x_0)=\alpha\quad\text{and}\quad u(y_0)=\beta.
$$
Consequently, the dimension of the space of extremals is $2n+2$.  Furthermore, 
$$
\int_{\R^n}|Du|^pdx\le \int_{\R^n}|Dv|^pdx
$$
for every $v:\R^n\rightarrow\R$ which satisfies the pointwise constraints 
$$
v(x_0)=\alpha\quad\text{and}\quad v(y_0)=\beta.
$$
This observation allows us to numerically approximate extremals with coordinate gradient descent (as discussed in Appendix \ref{AppNum}).  
See Figure \ref{ExtremalFig}. 
 %2D figure
\begin{figure}[h]
\centering
 \includegraphics[width=1\columnwidth]{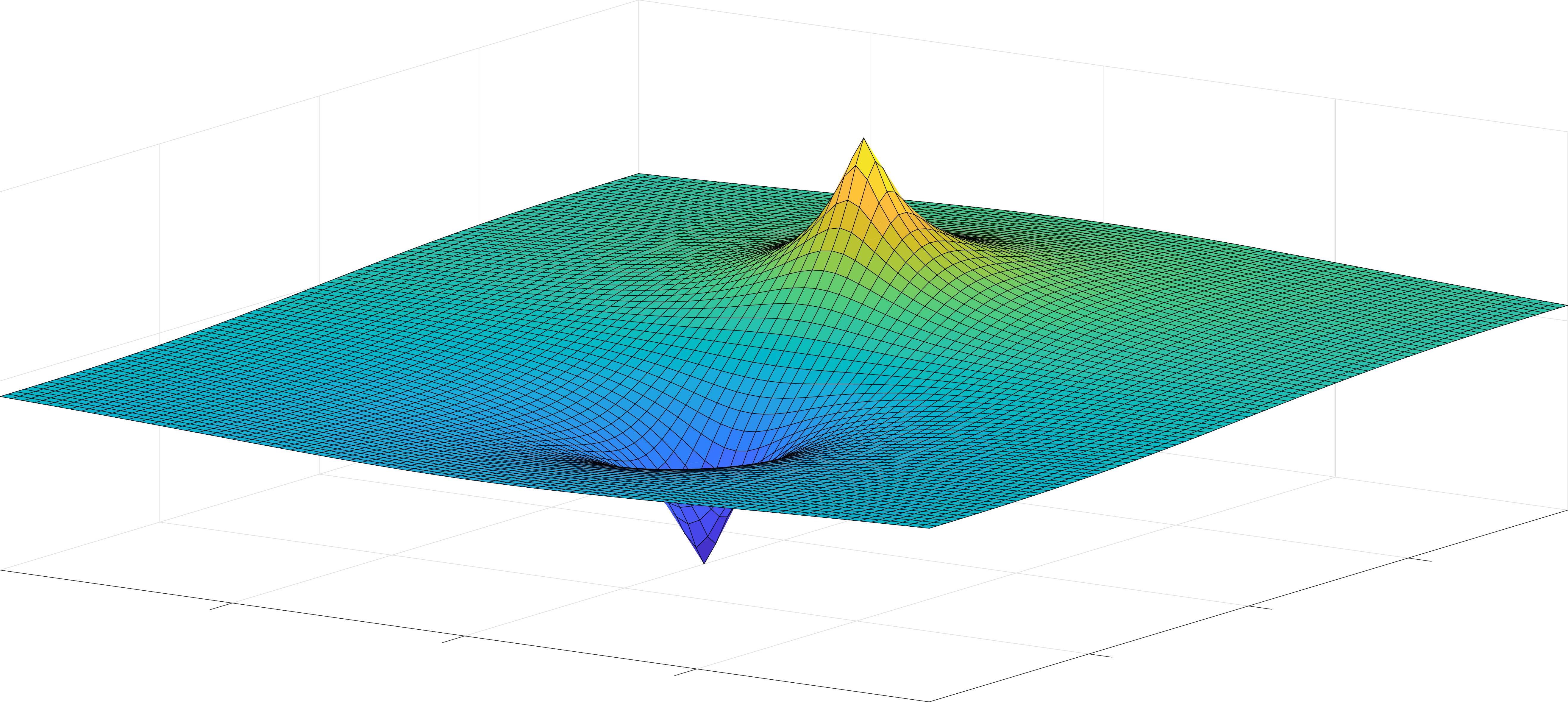}
 \caption{The graph of an extremal for Morrey's inequality in two spatial dimensions.}\label{ExtremalFig}
\end{figure}
\\
\par {\bf Cylindrical symmetry}. Assume $u$ is an extremal which satisfies \eqref{MaxHoldeRatio}. Then 
$$
u(x)=u(O(x-x_0)+x_0)
$$
for any orthogonal transformation $O$ of $\R^n$ which satisfies 
$$
O(y_0-x_0)=y_0-x_0.
$$
That is, $u$ is invariant under rigid transformations of $\R^n$ that fix the line passing through $x_0$ and $y_0$.  In particular, extremals are cylindrical about this line and are not radially symmetric. 
\\
\par {\bf Reflectional antisymmetry}.  If $u$ is an extremal whose H\"older ratio is maximized at distinct $x_0,y_0\in \R^n$, then
$$
u\left(x -2\frac{\left((x_0-y_0)\cdot (x-\frac{1}{2}(x_0+y_0)\right)}{|x_0-y_0|^2}(x_0-y_0)\right)-\frac{u(x_0)+u(y_0)}{2}
=-\left(u(x)-\frac{u(x_0)+u(y_0)}{2}\right).
$$
In other words, the function 
$$
u-\frac{1}{2}(u(x_0)+u(y_0))
$$
is antisymmetric with respect to reflection about the hyperplane with normal $x_0-y_0$ that contains $\frac{1}{2}(x_0+y_0)$.
\\
\par {\bf Boundedness}. Suppose $u$ is an extremal which satisfies \eqref{MaxHoldeRatio}. Then $u$ is bounded with 
$$
\inf_{x\in\R^n}u(x)=\min\{u(x_0),u(y_0)\}
$$
and
$$
\sup_{x\in\R^n}u(x)=\max\{u(x_0),u(y_0)\}.
$$
In particular, $u$ assumes its minimum and maximum values at points which maximize its H\"older seminorm. Furthermore, if $n\ge 2$ and $u$ is nonconstant, $u$ attains its maximum and minimum values uniquely at $x_0$ and $y_0$. 
\\
\par {\bf Quasiconcavity}. Suppose that $u$ is an extremal whose H\"older ratio is maximized at distinct $x_0,y_0\in \R^n$.  Also consider the two half-spaces 
$$
\Pi_\pm=\left\{x\in \R^n: \pm \left(x-\frac{1}{2}(x_0+y_0)\right)\cdot(x_0-y_0)>0\right\}
$$ 
divided by hyperplane with normal $x_0-y_0$ that contains $\frac{1}{2}(x_0+y_0)$.
If $u(x_0)>u(y_0)$, then
$$
\mbox{ $u|_{\Pi_+}$ is quasiconcave and  $u|_{\Pi_-}$ is quasiconvex};
$$
alternatively if $u(x_0)<u(y_0)$, then
$$
\mbox{ $u|_{\Pi_+}$ is quasiconvex and  $u|_{\Pi_-}$ is quasiconcave}.
$$
In either case, each level set of $u$ is the boundary of a convex subset of $\R^n$.  Refer to Figure \ref{ContourExt}. 
\\
\par {\bf Nonvanishing gradient}. Suppose $n\ge 2$. If $u$ is a nonconstant extremal whose H\"older ratio is maximized at distinct $x_0,y_0\in \R^n$, then  
$$
|Du|>0\quad \text{in}\quad \R^n\setminus\{x_0,y_0\}.
$$
This positivity combined with equation \eqref{strongPDE} will imply that $u$ has continuous derivatives of all orders and is in fact real analytic on $\R^n\setminus\{x_0,y_0\}$.  It also implies there are no points where $u$ has a local minimum or maximum except for $x_0$ and $y_0$. In addition, this property will allow us to show that $C_*$ cannot be found using the typical method employed to establish Morrey's inequality \eqref{MorreyIneq} for some constant $C$.
\\
% Organization and acknowledgements 
\par We will begin by indicating the natural invariances associated with Morrey's inequality and by exploiting these invariances to establish the existence of extremal functions. Then we will proceed to verify each of the above assertions roughly in the order that they are presented with the exception of property \eqref{MaxHoldeRatio}; we will save this technical assertion for the last section of this paper.  Finally, we would like to thank Eric Carlen, Yat Tin Chow, Elliott Lieb, Bob Kohn, 
Erik Lindgren, Stan Osher, and Tak Kwong Wong for their advice and insightful discussions related to this work.

%%%%%%%%%%%%%%%%%%%%%%%%%%%%%%%%%%%%%%%%%%%%%%%%%%%%%%%%%%%%%%%%%%%%%%%%%%%%%%%%

\section{Preliminaries}
% Notation
\par In what follows, it will be convenient for us to adopt the notation 
\be
[u]_{C^{1-n/p}(\R^n)}:=\sup_{x\ne y}\left\{\frac{|u(x)-u(y)|}{|x-y|^{1-n/p}}\right\}
\ee
and 
\be
\|Du\|_{L^p(\R^n)}:=\left(\int_{\R^n}|Du|^pdx\right)^{1/p}.
\ee
This allows us to restate Morrey's inequality \eqref{MorreyIneq} more concisely as 
$$
[u]_{C^{1-n/p}(\R^n)}\le C\|Du\|_{L^p(\R^n)}.
$$
For the remainder of this paper, we will also reserve $C$ for any constant such that Morrey's inequality holds and $C_*$ for the sharp constant. 

% Class of functions
\par  A natural class of functions for us to consider is the homogeneous Sobolev space
\be
{\cal  D}^{1,p}(\R^n):=\left\{u\in L^1_{\text{loc}}(\R^n): u_{x_i}\in L^p(\R^n)\;\text{for $i=1,\dots, n$}\right\}.
\ee 
In particular, $u\in {\cal  D}^{1,p}(\R^n)$ means that $u$ is weakly differentiable and all of its weak first order partial derivatives are $L^p(\R^n)$ functions.  What's more is that we can interpret Morrey's inequality to hold on ${\cal  D}^{1,p}(\R^n)$ once we recall that each $u\in {\cal  D}^{1,p}(\R^n)$ has a H\"older continuous representative $u^*:\R^n\rightarrow \R$. Namely, $u(x)=u^*(x)$ for Lebesgue almost every $x\in \R^n$ and
\be\label{RealMorreyIneq}
[u^*]_{C^{1-n/p}(\R^n)}\le C\|Du\|_{L^p(\R^n)}.
\ee
As such an argument is now routine within the theory of Sobolev spaces, we omit the details. Moreover, we will always identify $u$ with $u^*$ going forward.

%-------------------------------------------------------------------------------------------
\subsection{Existence of a nonconstant extremal}
% Invariances
Note that the seminorms $u\mapsto [u]_{C^{1-n/p}(\R^n)}$ and $u\mapsto  \|Du\|_{L^p(\R^n)}$ are invariant under the following transformations. 

\begin{itemize}

\item Multiplication by $-1$: $u(x)\mapsto -u(x)$

\item Addition by a constant $c\in \R$: $u(x)\mapsto u(x)+c$

\item Translation by a vector $a\in \R^n$: $u(x)\mapsto u(x+a)$

\item Application of an orthogonal transformation $O$ of $\R^n$: $u(x)\mapsto u(Ox)$

\item Scaling by $\lambda>0$: $u(x)\mapsto \lambda^{n/p-1}u(\lambda x)$

\end{itemize}
We will exploit these invariances to verify that extremals for Morrey's inequality exist. 
% Existence
\begin{lem}\label{ExistenceLem}
There exists a nonconstant $u\in {\cal  D}^{1,p}(\R^n)$ such that 
\be\label{RealMorreyIneq}
[u]_{C^{1-n/p}(\R^n)}= C_*\|Du\|_{L^p(\R^n)}.
\ee
In particular, $C_*>0$.
\end{lem}
\begin{proof}
Let us define
\be
\Lambda:=\inf\left\{\|Du\|_{L^p(\R^n)} \; : u\in {\cal  D}^{1,p}(\R^n),\; [u]_{C^{1-n/p}(\R^n)}=1 \right\}
\ee
and also choose a minimizing sequence $(u_k)_{k\in \N}$ for which
\be
\Lambda=\lim_{k\rightarrow\infty}\|Du_k\|_{L^p(\R^n)}.
\ee
For each $k\in \N$, we can select $x_k,y_k\in \R^n$ with $x_k\neq y_k$ such that
\be
1=[u_k]_{C^{1-n/p}(\R^n)}<\frac{u_k(x_k)-u_k(y_k)}{|x_k-y_k|^{1-n/p}}+\frac{1}{k}.
\ee
We may also find an orthogonal transformation $O_k$ satisfying 
$$
O_ke_n=\frac{y_k-x_k}{|x_k-y_k|},
$$ 
where $e_n=(0,\dots, 0,1)$. 

\par In addition, we define
$$
v_k(z):=|x_k-y_k|^{n/p-1}\Big\{u_k\left(|x_k-y_k| O_kz+x_k\right)-u_k(x_k)\Big\}
$$
for $z\in \R^n$ and $k\in \N$.  By the invariances of the seminorms $u\mapsto [u]_{C^{1-n/p}(\R^n)}$ and $u\mapsto  \|Du\|_{L^p(\R^n)}$,  
$$
[v_k]_{C^{1-n/p}(\R^n)}=1 \quad \text{and}\quad \Lambda=\lim_{k\in \N}\|Dv_k\|_{L^p(\R^n)}.
$$
Moreover, 
$$
v_k(0)=0\quad \text{and}\quad  1-\frac{1}{k}< v_k(e_n)\le 1. 
$$
We can now employ a standard variant of the Arzel\`a-Ascoli theorem to obtain a subsequence $(v_{k_j})_{j\in \N}$ converging locally uniformly 
to a continuous function $v:\R^n\rightarrow \R$. 

\par Local uniform convergence immediately implies
$$
v(0)=0,\quad v(e_n)= 1, \quad  \text{and}\quad [v]_{C^{1-n/p}(\R^n)}\le 1. 
$$
As
$$
1=\frac{v(e_n)-v(0)}{|e_n-0|^{1-n/p}}\le [v]_{C^{1-n/p}(\R^n)},
$$
we actually have
$$
[v]_{C^{1-n/p}(\R^n)}= 1.
$$ 
Since $\|Dv_k\|_{L^p(\R^n)}$ is bounded, we may as well suppose that $(Dv_{k_j})_{j\in \N}$ converges weakly in $L^p(\R^n;\R^n)$.
Using the local uniform convergence of $v_{k_j}$, it is easy to verify that the weak limit of $D v_{k_j}$ in $L^p (\R^n;\R^n)$ is the weak derivative of $v$.  It then follows that $v\in {\cal  D}^{1,p}(\R^n)$ and
$$
\Lambda=\liminf_{j\in \N}\|Dv_{k_j}\|_{L^p(\R^n)}\ge \|Dv\|_{L^p(\R^n)}.
$$

\par Since $v$ is nonconstant, $\Lambda$ is positive. By the definition of $\Lambda$, $C=1/\Lambda$ is a constant for which Morrey's inequality holds. 
As
$$
1=[v]_{C^{1-n/p}(\R^n)}\le \frac{1}{\Lambda}\|Dv\|_{L^p(\R^n)}\le 1,
$$
$1/\Lambda=C_*>0$ and $v$ is a nonconstant extremal of Morrey's inequality. 
\end{proof}
\begin{cor}\label{ExistenceCor}
Assume $x_0,y_0\in \R^n$ and $\alpha,\beta\in \R$ are distinct. There is an extremal $u$ which satisfies $u(x_0)=\alpha$ and $u(y_0)=\beta$ and whose $1-n/p$ H\"older ratio is maximized at $x_0$ and $y_0$.
\end{cor}
\begin{proof}
In the proof of the above lemma, we showed there is an extremal $v$ which satisfies $v(0)=0$ and $v(e_n)=1$ and whose $1-n/p$ H\"older ratio is maximized at $0$ and $e_n$. We now select an orthogonal transformation $O$ of $\R^n$ for which 
$$
O\left(\frac{y_0-x_0}{|y_0-x_0|}\right)=e_n
$$ 
and set
$$
u(x)=(\beta-\alpha)v\left(\frac{O\left(x-x_0\right)}{|y_0-x_0|}\right)+\alpha
$$
for $x\in \R^n$.  Clearly $u(x_0)=\alpha$ and $u(y_0)=\beta$, and by the scaling invariance of the H\"older seminorm
$$
[u]_{C^{1-n/p}(\R^n)}=\frac{|\beta-\alpha|}{|x_0-y_0|^{1-n/p}}=\frac{|u(x_0)-u(y_0)|}{|x_0-y_0|^{1-n/p}}.
$$
\end{proof}

%------------------------------------------------------------------------------------------------------------
\subsection{PDE for extremals} 
We have established the existence of an extremal whose
$1-n/p$ H\"older ratio attains its maximum at a pair of distinct points.  We will later show that every function in ${\cal  D}^{1,p}(\R^n)$ has this property. For now, we will assume this property and use it to derive a PDE satisfied by extremal functions. 

% Weak form
\begin{prop}\label{PDEprop}
Suppose $u\in {\cal  D}^{1,p}(\R^n)$ is an extremal whose $1-n/p$ H\"older ratio attains its maximum at two distinct points $x_0,y_0\in \R^n$. Then for each $\phi \in {\cal  D}^{1,p}(\R^n)$
\be\label{weakPDE}
C^p_*\int_{\R^n}|Du|^{p-2}Du\cdot D\phi dx=\frac{|u(x_0)-u(y_0)|^{p-2}(u(x_0)-u(y_0))}{|x_0-y_0|^{p-n}}(\phi(x_0)-\phi(y_0)).
\ee
\end{prop}
\begin{proof}
By assumption,
\be\label{PDEderiv1}
\frac{|u(x_0)-u(y_0)|^p}{|x_0-y_0|^{p-n}}=C_*^p\int_{\R^n}|Du|^pdx;
\ee
and in view of Morrey's inequality, 
\be\label{PDEderiv2}
\frac{|u(x_0)-u(y_0)+t(\phi(x_0)-\phi(y_0))|^p}{|x_0-y_0|^{p-n}}\le C_*^p\int_{\R^n}|Du+tD\phi|^pdx
\ee
for each $\phi \in {\cal  D}^{1,p}(\R^n)$ and $t>0$. Moreover, the convexity of the function $\R\ni w\mapsto |w|^p$ can be used to
derive
\begin{align}\label{PDEderiv3}
\frac{|u(x_0)-u(y_0)+t(\phi(x_0)-\phi(y_0))|^p}{|x_0-y_0|^{p-n}}&\ge \frac{|u(x_0)-u(y_0)|^p}{|x_0-y_0|^{p-n}} +\\
& \quad t p\frac{|u(x_0)-u(y_0)|^{p-2}(u(x_0)-u(y_0))}{|x_0-y_0|^{p-n}}(\phi(x_0)-\phi(y_0)).
\end{align}
As a result, we can subtract \eqref{PDEderiv1} from \eqref{PDEderiv2} and use \eqref{PDEderiv3} to obtain
\be\label{PDEderiv4}
 C_*^p\displaystyle\int_{\R^n}\left(\frac{|Du+tD\phi|^p-|Du|^p}{pt}\right)dx\ge \frac{|u(x_0)-u(y_0)|^{p-2}(u(x_0)-u(y_0))}{|x_0-y_0|^{p-n}}(\phi(x_0)-\phi(y_0)).
 \ee
 
 \par It is also possible to find a constant $c_p$ such that \be
0\le \frac{|Du+tD\phi|^p-|Du|^p}{pt}- |Du|^{p-2}Du\cdot D\phi
\le c_p
\begin{cases}
t^{p-1} |D\phi|^{p}, \quad &1<p<2\\
t |D\phi|^2(|Du|+|D\phi|)^{p-2}, \quad &2\le p<\infty
\end{cases}
\ee
for $t\in (0,1]$; see for example Lemma 10.2.1 in \cite{AGS}. With this estimate, we can pass to the limit as $t\rightarrow 0^+$ in \eqref{PDEderiv4} to conclude
\be\label{AlmostweakPDE} 
C^p_*\int_{\R^n}|Du|^{p-2}Du\cdot D\phi dx\ge \frac{|u(x_0)-u(y_0)|^{p-2}(u(x_0)-u(y_0))}{|x_0-y_0|^{p-n}}(\phi(x_0)-\phi(y_0)).
\ee
We arrive at \eqref{weakPDE} once we replace $\phi$ with $-\phi$ in \eqref{AlmostweakPDE}. 
\end{proof}

% PDE
\par If $u\in {\cal  D}^{1,p}(\R^n)$ satisfies \eqref{weakPDE} for each $\phi\in {\cal  D}^{1,p}(\R^n)$, we can choose $\phi=u$ to find that $u$ is a 
necessarily an extremal function for Morrey's inequality and the $1-n/p$ H\"older seminorm of $u$ is attained at $(x_0,y_0)$.  Also note that \eqref{weakPDE} implies that $u$ is a weak solution of the PDE \eqref{strongPDE}
\be
-\Delta_pu=\frac{|u(x_0)-u(y_0)|^{p-2}(u(x_0)-u(y_0))}{C_*^p|x_0-y_0|^{p-n}}(\delta_{x_0}-\delta_{y_0})
\ee
in $\R^n$.  In particular, $u$ is $p$-harmonic
$$
-\Delta_pu=0
$$
in $\R^n\setminus\{x_0,y_0\}$. As a result, $Du$ is a locally H\"older continuous mapping from $\R^n\setminus\{x_0,y_0\}$ into $\R^n$ \cite{Evans, Lew, Ural}.

% Equivalence
\par We will use the weak form of \eqref{strongPDE} to deduce various properties of extremals below.  Our first observation is that there are at least 
two other ways to identify that a function $u\in {\cal  D}^{1,p}(\R^n)$ is an extremal. The first is as a solution of a PDE of the type 
\eqref{strongPDE}; the second is as a minimizer of a related convex optimization problem. 
\begin{thm}\label{equivalentChar}
Suppose $x_0,y_0\in \R^n$ are distinct and $u\in {\cal  D}^{1,p}(\R^n)$.  Then the following are equivalent. 

\begin{enumerate}[(i)]

\item $u$ is an extremal with 
$$
[u]_{C^{1-n/p}(\R^n)}=\frac{|u(x_0)-u(y_0)|}{|x_0-y_0|^{1-n/p}}.
$$

\item There is a $c\in \R$ for which $u$ satisfies 
\be\label{TwoDeltaPDE}
-\Delta_pu=c(\delta_{x_0}-\delta_{y_0})
\ee
in $\R^n$.

\item For each $v\in {\cal  D}^{1,p}(\R^n)$ with $v(x_0)=u(x_0)$ and $v(y_0)=u(y_0)$,
\be\label{vLowerEniii}
\int_{\R^n}|Du|^pdx\le \int_{\R^n}|Dv|^pdx.
\ee
\end{enumerate}

\end{thm}
\begin{rem} In $(ii)$, \eqref{TwoDeltaPDE} means 
\be\label{WeakTwoDeltaPDE}
\int_{\R^n}|Du|^{p-2}Du\cdot D\phi dx=c(\phi(x_0)-\phi(y_0))
\ee
for each $\phi\in {\cal  D}^{1,p}(\R^n)$. 
\end{rem}

% Proof of theorem
\begin{proof}
We have already showed that $(i)$ implies $(ii)$ in Proposition \ref{PDEprop}. Let us assume $(ii)$, and suppose $v\in {\cal  D}^{1,p}(\R^n)$ satisfies $v(x_0)=u(x_0)$ and $v(y_0)=u(y_0)$. Applying \eqref{WeakTwoDeltaPDE}, we find
\begin{align*}
\int_{\R^n}|Dv|^pdx&\ge \int_{\R^n}|Du|^pdx+p\int_{\R^n}|Du|^{p-2}Du\cdot (Dv-Du)dx\\
&=\int_{\R^n}|Du|^pdx+pc((v-u)(x_0)-(v-u)(y_0))\\
&=\int_{\R^n}|Du|^pdx.
\end{align*}
Therefore, $(ii)$ implies $(iii)$. 
%\par Now suppose that $u$ satisfies $(iii)$ and let $\phi\in {\cal  D}^{1,p}(\R^n)$ with $\phi(x_0)=\phi(y_0)=0$. Then 
%\begin{align*}
%\int_{\R^n}|Du|^pdx&\le \int_{\R^n}|Du+tD\phi|^pdx\\
%&=\int_{\R^n}|Du|^pdx+pt\int_{\R^n}|Du|^{p-2}Du\cdot D\phi dx +o(t)
%\end{align*}
%as $t\rightarrow 0$. It follows that 
%\be\label{WeakTwoDeltaPDE2}
%\int_{\R^n}|Du|^{p-2}Du\cdot D\phi dx=0.
%\ee
%\par Select $\phi=u-w$, where $w$ is an extremal which satisfies 
%\be\label{wHolderNorm}
%[w]_{C^{1-n/p}(\R^n)}=\frac{|w(x_0)-w(y_0)|}{|x_0-y_0|^{1-n/p}}
%\ee
%and
%\be\label{tildewandu}
% w(x_0)=u(x_0)\quad \text{and}\quad  w(y_0)=u(y_0).
%\ee
%We can now appeal to \eqref{WeakTwoDeltaPDE2} and Proposition \ref{PDEprop} to get
%$$
%\int_{\R^n}|Du|^{p-2}Du\cdot D(u- w)dx=0\quad \text{and}\quad 
%\int_{\R^n}|D w|^{p-2}D w\cdot D( w-u)dx=0.
%$$
%In particular, 
%$$
%\int_{\R^n}(|Du|^{p-2}Du-|D w|^{p-2}D w)\cdot D(u- w)dx=0.
%$$
%As the mapping $z\mapsto |z|^{p-2}z$ of $\R^n$ is strictly monotone, $Du\equiv D w$. In view of \eqref{tildewandu}, $u\equiv w $ and we conclude that $(iii)$ implies $(i)$.
%\end{proof}
\par Now suppose that $u$ satisfies $(iii)$ and $w$ is an extremal for which 
\be\label{wHolderNorm}
[w]_{C^{1-n/p}(\R^n)}=\frac{|w(x_0)-w(y_0)|}{|x_0-y_0|^{1-n/p}}
\ee
and
\be\label{tildewandu}
 w(x_0)=u(x_0)\quad \text{and}\quad  w(y_0)=u(y_0).
\ee
Choosing $v=u+t(w-u)$ in \eqref{vLowerEniii} gives
\begin{align*}
\int_{\R^n}|Du|^pdx&\le \int_{\R^n}|Du+tD(w-u)|^pdx\\
&=\int_{\R^n}|Du|^pdx+pt\int_{\R^n}|Du|^{p-2}Du\cdot D(w-u) dx +o(t)
\end{align*}
as $t\rightarrow 0$. It follows that 
\be\label{WeakTwoDeltaPDE2}
\int_{\R^n}|Du|^{p-2}Du\cdot D(w-u) dx=0.
\ee
By Proposition \ref{PDEprop}, we also have
$$
\int_{\R^n}|D w|^{p-2}D w\cdot D( w-u)dx=0.
$$
In particular, 
$$
\int_{\R^n}(|Du|^{p-2}Du-|D w|^{p-2}D w)\cdot D(u- w)dx=0.
$$
As the mapping $z\mapsto |z|^{p-2}z$ of $\R^n$ is strictly monotone, $Du\equiv D w$. In view of \eqref{tildewandu}, $u\equiv w $ and we conclude that $(iii)$ implies $(i)$.
\end{proof}

%--------------------------------------------------------------------------------------------
\subsection{Extremals in one spatial dimension}
% Extremals in one spatial dimension
Suppose $n=1$ and $u\in {\cal  D}^{1,p}(\R)$. As $u'\in L^p(\R)$, $u$ is absolutely continuous and 
\be\label{FunThmCalc}
|u(x)-u(y)|=\left|\int^x_yu'dz\right|\le \left(\int^x_y|u'|^pdz\right)^{1/p}|x-y|^{1-1/p}
\ee
for $y\le x$. As a result
\be\label{1DMorrey}
[u]_{C^{1-1/p}(\R)}\le \|u'\|_{L^p(\R)},
\ee
and Morrey's inequality holds with $C_*\le 1$.  Moreover, it is easy to check that equality holds for
\be\label{1Dextremalfun}
u(x)=
\begin{cases}
-1, \quad &x\in (-\infty,-1)\\
x, \quad & x\in [-1,1]\\
1, \quad & x\in (1,\infty).
\end{cases}
\ee
Therefore, $C_*=1$. We also note that the $1-1/p$ H\"older ratio of $u$ is maximized at $\pm1$. 

\par Conversely, if equality holds in \eqref{1DMorrey} for some $u\in {\cal  D}^{1,p}(\R)$ and if
\be\label{HolderMax1D}
[u]_{C^{1-1/p}(\R)}=\frac{|u(x_0)-u(y_0)|}{|x_0-y_0|^{1-1/p}}
\ee 
for some $y_0<x_0$, then \eqref{FunThmCalc} gives
$$
\left(\int^{x_0}_{y_0}|u'|^pdz\right)^{1/p}=\left(\int_{\R}|u'|^pdz\right)^{1/p}.
$$
It would then follow that $u'$ vanishes on $\R\setminus(y_0,x_0)$ and that $u'$ is necessarily constant on $[y_0, x_0]$.  
As a result, $u$ would be of the form \eqref{1Dextremalfun}. In the last section of this paper, we will verify \eqref{HolderMax1D}  and use 
a higher dimensional analog of \eqref{FunThmCalc} to verify that the H\"older ratio of each extremal is maximized at distinct points for $n\ge 2$. 

% 1D figure
\begin{figure}[h]
\centering
 \includegraphics[width=.65\textwidth]{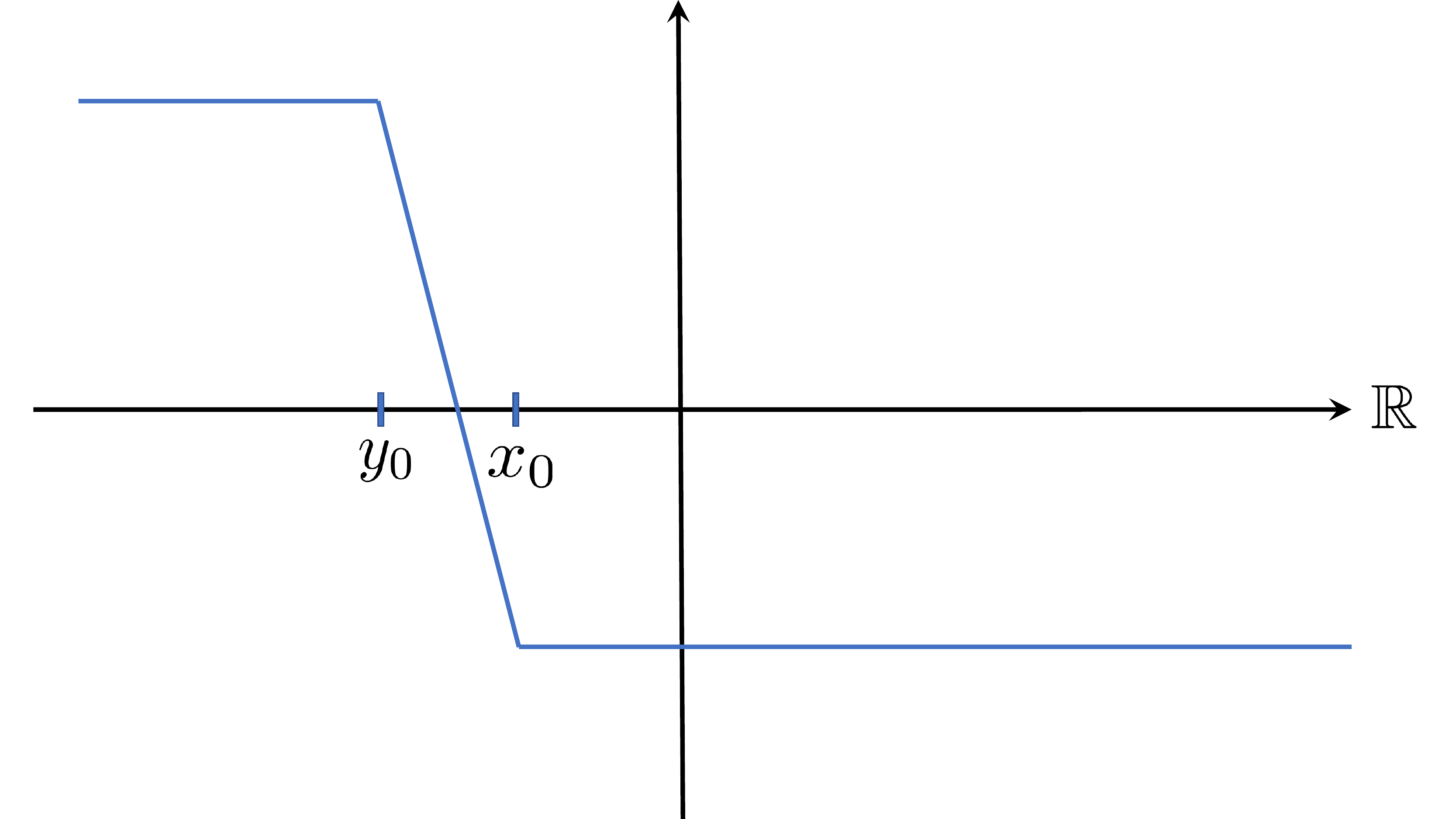}
 \caption{The graph of an extremal in one spatial dimension whose $1-1/p$ H\"older ratio is maximized at $x_0$ and $y_0$.}%\label{}
\end{figure}

%%%%%%%%%%%%%%%%%%%%%%%%%%%%%%%%%%%%%%%%%%%%%%%%%%%%%%%%%%%%%%%%%%%%%%%%%%%
\section{Symmetry and pointwise bounds}
We will now use PDE \eqref{strongPDE} to verify a variety of assertions for extremals.  We start off by showing that extremals are unique up to the natural invariances associated with Morrey's inequality.  
This uniqueness is employed to derive some symmetry and antisymmetry properties. Then we will use the strong maximum principle and Hopf's boundary point lemma for $p$-harmonic functions to bound extremals above and below and to obtain a sign of a directional derivative along an antisymmetry plane for extremals, respectively.

\subsection{Uniqueness}
In Corollary \ref{ExistenceCor}, we showed that for and distinct $\alpha,\beta\in\R$ and $x_0, y_0\in \R^n$, there is an extremal $u\in{\cal  D}^{1,p}(\R^n)$ that satisfies $u(x_0)=\alpha$ and $u(x_0)=\beta$ and whose H\"older ratio is maximized at $x_0$ and $y_0$. It turns out that this extremal is the only one that satisfies these constraints.  The key observation needed to justify this statement is as follows. 
\begin{prop}\label{UniqueProp}
Suppose $x_0, y_0\in \R^n$ and $x_1, y_1\in \R^n$ are distinct, and assume $u, v\in {\cal  D}^{1,p}(\R^n)$ are nonconstant extremals with 
$$
[u]_{C^{1-n/p}(\R^n)}=\frac{u(x_0)-u(y_0)}{|x_0-y_0|^{1-n/p}}\quad \text{and}\quad 
[v]_{C^{1-n/p}(\R^n)}=\frac{v(x_1)-v(y_1)}{|x_1-y_1|^{1-n/p}}.
$$
Then for each orthogonal transformation $O$ of $\R^n$ which satisfies 
$$
O\left(\frac{y_0-x_0}{|y_0-x_0|}\right)=\frac{y_1-x_1}{|y_1-x_1|}
$$
and each $x\in \R^n$, 
\be
u(x)=\frac{u(y_0)-u(x_0)}{v(y_1)-v(x_1)}\cdot \left\{v\left(\frac{|x_1-y_1|}{|x_0-y_0|} O(x-x_0)+x_1\right)-v(x_1)\right\}+u(x_0).
\ee
\end{prop}
\begin{proof}
Set 
\be
\tilde u(x):=\frac{u(y_0)-u(x_0)}{v(y_1)-v(x_1)}\cdot \left\{v\left(\frac{|x_1-y_1|}{|x_0-y_0|} O(x-x_0)+x_1\right)-v(x_1)\right\}+u(x_0)
\ee
for $x\in \R^n$. By design,
\be\label{TwoPoints}
\tilde u(x_0)=u(x_0)\quad \text{and}\quad \tilde u(y_0)=u(y_0).
\ee
And in view of the invariances of the seminorms associated with Morrey's inequality, $\tilde u$ is a nonconstant extremal with 
$$
[\tilde u]_{C^{1-n/p}(\R^n)}=\frac{\tilde u(x_0)-\tilde u(y_0)}{|x_0-y_0|^{1-n/p}}.
$$
Consequently, both $u$ and $\tilde u$ are both weak solutions of the PDE \eqref{strongPDE}.  

\par Evaluating \eqref{weakPDE} at $\phi=u-\tilde u$ gives
\begin{align*}
\int_{\R^n}|D u|^{p-2}D u\cdot (Du-D\tilde u) dx=0,
\end{align*}
and similarly, 
\be
\int_{\R^n}|D \tilde u|^{p-2}D \tilde u\cdot (D u-D\tilde u) dx=0.
\ee
Subtracting these equalities gives 
$$
\int_{\R^n}(|D u|^{p-2}D u- |D\tilde u|^{p-2}D\tilde u)\cdot (D u - D\tilde u)dx=0.
$$
Therefore, $u-\tilde u$ is constant and by \eqref{TwoPoints} must in fact vanish identically. 
\end{proof}
\begin{cor}\label{TrueUnique}
Suppose $u_1, u_2\in {\cal  D}^{1,p}(\R^n)$ are nonconstant extremals whose $1-n/p$ H\"older seminorms are both maximized at a pair of distinct points $x_0,y_0\in\R^n$. If in addition  
$$
u_1(x_0)=u_2(x_0)\quad \text{and}\quad u_1(y_0)=u_2(y_0),
$$
then $u_1(x)=u_2(x)$ for all $x\in \R^n$. 
\end{cor}
\begin{proof}
We can assume that $u_1(x_0)>u_1(y_0)$, or else we can prove this corollary for $-u_1$ and $-u_2$.  The assertion is then immediate once we realize we can select $O=I_n$ in the statement of Proposition \ref{UniqueProp} with $u=u_1$ and $v=u_2$.
\end{proof}
A closer inspection of the previous assertion implies that any extremal $u$ whose $1-n/p$ H\"older ratio achieves its maximum at a pair of distinct points possesses a strong 
symmetry property.  Specifically, $u$ will be cylindrically symmetric and its axis of symmetry is the line through the points at which its $1-n/p$ H\"older seminorm is achieved.
\begin{cor}
Suppose $u\in {\cal  D}^{1,p}(\R^n)$ is an extremal with 
$$
[u]_{C^{1-n/p}(\R^n)}=\frac{|u(x_0)-u(y_0)|}{|x_0-y_0|^{1-n/p}}.
$$
Then
\be
u(x)=u(O (x-x_0)+x_0), \quad x\in \R^n
\ee
for any orthogonal transformation $O$ which satisfies 
$$
O\left(y_0-x_0\right)=y_0-x_0.
$$
\end{cor}
\begin{proof}
Set
$$
v(x):=u(O (x-x_0)+x_0), \quad x\in \R^n. 
$$
The claim follows once we observe that $v$ is an extremal with $[v]_{C^{1-n/p}(\R^n)}=[u]_{C^{1-n/p}(\R^n)}$, $v(x_0)=u(x_0)$, and $v(y_0)=u(y_0)$.
\end{proof}
Extremals also have a certain antisymmetry property.  We recall that the orthogonal reflection about the hyperplane $\{x\in\R^n: x\cdot a=c\}$ 
is given by the transformation
$$
x\mapsto x -2\frac{(x\cdot a-c)}{|a|^2}a.
$$
Note in particular that this mapping is a composition of an orthogonal transformation and a translation.  We will show below that $u-\frac{1}{2}(u(x_0)+u(y_0))$ is antisymmetric about the hyperplane with normal pointing in the same direction as $x_0-y_0$ and that passes through the midpoint of $x_0$ and $y_0$. 
\begin{prop}\label{reflectionProp}
Suppose $u\in {\cal  D}^{1,p}(\R^n)$ is an extremal with 
$$
[u]_{C^{1-n/p}(\R^n)}=\frac{|u(x_0)-u(y_0)|}{|x_0-y_0|^{1-n/p}}.
$$
Then 
$$
u\left(x -2\frac{\left((x_0-y_0)\cdot (x-\frac{1}{2}(x_0+y_0)\right)}{|x_0-y_0|^2}(x_0-y_0)\right)-\frac{u(x_0)+u(y_0)}{2}
=-\left(u(x)-\frac{u(x_0)+u(y_0)}{2}\right)
$$
for each $x\in \R^n$. 
\end{prop}
\begin{proof}
Define
$$
v(x):=u(x_0)+u(y_0)-u\left(x -2\frac{\left((x_0-y_0)\cdot (x-\frac{1}{2}(x_0+y_0)\right)}{|x_0-y_0|^2}(x_0-y_0)\right)
$$
We note that  
$$
x\mapsto x -2\frac{\left((x_0-y_0)\cdot (x-\frac{1}{2}(x_0+y_0)\right)}{|x_0-y_0|^2}(x_0-y_0)
$$
is orthogonal reflection about the hyperplane with normal $x_0-y_0$ which passes through $(x_0+y_0)/2$. 
Since this map is a composition of an orthogonal transformation and a translation, $v$ is an extremal with $[u]_{C^{1-n/p}(\R^n)}=[v]_{C^{1-n/p}(\R^n)}$.  As 
$$
v(x_0)=u(x_0)\quad\text{and}\quad v(y_0)=u(y_0),
$$
$v(x)=u(x)$ for all $x\in \R^n$ by Corollary \ref{TrueUnique}. 
\end{proof}

%-------------------------------------------------------------------------------------------------------------------
\subsection{Pointwise bounds}
As mentioned above, we will argue that each extremal is uniformly bounded. In particular, if $u\in {\cal  D}^{1,p}(\R^n)$ is an extremal with 
$$
[u]_{C^{1-n/p}(\R^n)}=\frac{|u(x_0)-u(y_0)|}{|x_0-y_0|^{1-n/p}},
$$
we will show that 
\be\label{PointwiseBounds}
\min\{u(x_0),u(y_0)\}\le u(x)\le \max\{u(x_0),u(y_0)\}
\ee
for all $x\in \R^n$.   When $n=1$, these inequalities are immediate in view of the explicit extremal \eqref{1Dextremalfun}.  Consequently, we will focus on the case $n\ge 2$ and also prove a refinement of \eqref{PointwiseBounds}.  
\begin{prop}
Suppose $n\ge 2$, $u\in {\cal  D}^{1,p}(\R^n)$ is an extremal with 
$$
[u]_{C^{1-n/p}(\R^n)}=\frac{u(x_0)-u(y_0)}{|x_0-y_0|^{1-n/p}}>0,
$$
and set
$$
\Pi_\pm:=\left\{x\in \R^n: \pm \left(x-\frac{1}{2}(x_0+y_0)\right)\cdot(x_0-y_0)>0\right\}.
$$ 
Then 
\be\label{ubetweenzeroandone}
\frac{u(x_0)+u(y_0)}{2}<u(x)<u(x_0)
\ee
for $ x\in \Pi_+$ and 
\be\label{ubetweenminusoneandzero}
u(y_0)<u(x)<\frac{u(x_0)+u(y_0)}{2}
\ee
for $ x\in \Pi_-$.
\end{prop}
\begin{proof}
We will only prove \eqref{ubetweenzeroandone} as a similar argument can be used to justify \eqref{ubetweenminusoneandzero}.   Moreover, we will suppose $x_0=e_n$, $y_0=-e_n$, $u( e_n)= 1$, and $u(-e_n)= -1$.  Here, $\Pi_{+}=\{x\in \R^n: x_n>0\}$. The general case would then follow by an appropriate change of variables. 

\par Observe that Proposition \ref{reflectionProp} implies 
$$
u|_{\partial\Pi_+}=0.
$$
We claim that 
\be\label{wClaimComp}
\int_{\Pi_+}|Du|^pdx\le \int_{\Pi_+}|Dw|^pdx 
\ee
for each $w\in  {\cal  D}^{1,p}(\R^n)$ which satisfies 
\be\label{wClaimConst}
w|_{\partial\Pi_+}=0\quad \text{and}\quad w(e_n)=1
\ee
and that equality holds in \eqref{wClaimComp} if and only if $w|_{\Pi_+}=u|_{\Pi_+}$. 

\par For a given $w\in  {\cal  D}^{1,p}(\R^n)$ satisfying \eqref{wClaimConst}, we can set
$$
v(x)=
\begin{cases}
w(x), & x\in\Pi_+\\
-w\left(x -2x_ne_n\right), & x\not\in\Pi_+ 
\end{cases}
$$
and check that $v\in {\cal  D}^{1,p}(\R^n)$. As 
$$
v(e_n)=w(e_n)=1\quad \text{and}\quad v(-e_n)=-w(e_n)=-1,
$$ 
Theorem \ref{equivalentChar} gives
$$
2\int_{\Pi_+}|Du|^pdx=\int_{\R^n}|Du|^pdx\le \int_{\R^n}|Dv|^pdx=2\int_{\Pi_+}|Dw|^pdx.
$$
This proves \eqref{wClaimComp}. If equality holds, then $v$ is an extremal (by Theorem \ref{equivalentChar}) which is necessarily equal to $u$ by Corollary \ref{TrueUnique}. 

\par Let us now choose 
$$
w(x)=\min\{u(x),1\}, \quad x\in\R^d.
$$
As $w$ fulfills the constraints \eqref{wClaimConst}, \eqref{wClaimComp} gives 
$$
\int_{\Pi_+}|Du|^pdx\le \int_{\Pi_+}|Dw|^pdx=\int_{\Pi_+\cap\{u\le 1\}}|Du|^pdx\le \int_{\Pi_+}|Du|^pdx.
$$
Consequently, 
$$
u(x)=w(x)=\min\{u(x),1\}\le 1
$$
for each $x\in \Pi_+$.  

\par Since $u$ is $p$-harmonic in the  domain $\Pi_+\setminus\{e_n\}$, either $u<1$ in 
or $u\equiv 1$ throughout $\Pi_+\setminus\{e_n\}$ by the strong maximum principle (Chapter 2 of \cite{MR2242021}). Since $u|_{\partial\Pi_+}=0$, it must be that $u<1$ in $\Pi_+\setminus\{e_n\}$.  Analogously, we can employ $w(x)=\max\{u(x),0\}$ for $x\in \Pi_+$ to deduce $u>0$ in  $\Pi_+$. We leave the details to the reader. 
\end{proof}

We finally assert that a certain directional derivative of extremals always has a sign. This observation will be useful when we discuss 
the analyticity of extremal functions. 
\begin{prop}
Assume $n\ge2$, $u\in {\cal  D}^{1,p}(\R^n)$ is an extremal with 
$$
[u]_{C^{1-n/p}(\R^n)}=\frac{u(x_0)-u(y_0)}{|x_0-y_0|^{1-n/p}}>0
$$
and set
$$
\Pi_+:=\left\{x\in \R^n: \left(x-\frac{1}{2}(x_0+y_0)\right)\cdot(x_0-y_0)>0\right\}.
$$ 
Then 
\be\label{NormalDerivativeu}
Du(x)\cdot (x_0-y_0)>0
\ee
for $ x\in \partial\Pi_+$.  
\end{prop}
\begin{proof}
Without loss of generality, we can focus on the case where $x_0=e_n$, $y_0=-e_n$, $u(e_n)=1$, and $u(-e_n)=-1$. In this case,  $\Pi_+=\{x\in \R^n: x_n>0\}$. 
By the previous proposition, $u>0$ in $\Pi_+$ and $u|_{\partial\Pi_+}=0$. Observe that for each $x\in \partial\Pi_+$, $u$ is $p$-harmonic in the ball
$$
D_x:=B_{\frac{1}{2}}\left(x+\frac{1}{2}e_n\right)
$$ 
and 
$$
u(x)=0<u(y),\quad y\in \overline{D_x}\setminus\{x\}.
$$
The latter observation follows from the fact that $D_x\subset \Pi_+$ and $\overline{D_x}\cap \partial\Pi_+=\{x\}$. By Hopf's boundary point lemma for $p$-harmonic functions (Lemma A.3 of \cite{MR951227}), 
$$
Du(x)\cdot (-e_n)<0. 
$$
\end{proof}

%%%%%%%%%%%%%%%%%%%%%%%%%%%%%%%%%%%%%%%%%%%%%%%%%%%%%%%%%%%%%%%
\section{Regularity}
In this section, we will consider the smoothness properties of extremals.  In particular, we will establish that extremals are smooth except for 
at the two points which maximize their H\"older seminorms. We also verify that their level sets bound convex regions in $\R^n$

% Behavior near singular points 
%--------------------------------------------------------------------------------------------------
\subsection{Nondifferentiability points}
It will be useful for us to recall some properties of $p$-harmonic functions in punctured domains. In particular,  the behavior of these functions 
near their singular points has been deduced by Kichenassamy and Veron \cite{KicVer,MR886428}. We will adapt their results to our setting as follows. 

\begin{lem}\label{singularPointLem}
Suppose $n\ge 2$ and $\Omega\subset\R^n$ is a bounded domain with $x_0\in \Omega$. Further suppose that $u\in {\cal  D}^{1,p}(\R^n)$ is $p$-harmonic in $\Omega\setminus\{x_0\}$ with $u(x)\le u(x_0)$ for $x\in\Omega$. Then
$$
\lim_{x\rightarrow x_0}\frac{u(x_0)-u(x)}{|x-x_0|^{\frac{p-n}{p-1}}}=\gamma
%\quad \text{and}\quad \lim_{x\rightarrow x_0}\frac{|Du(x)|}{|x-x_0|^{\frac{p-n}{p-1}-1}}=\left(\frac{p-n}{p-1}\right)\gamma
$$
for some $\gamma\ge 0$, and
$$
-\Delta_pu=n\omega_n\left(\frac{p-n}{p-1}\gamma\right)^{p-1}\delta_{x_0}
$$
in $\Omega$. 
\end{lem}
\begin{rem}
Here $\omega_n$ is the Lebesgue measure of the unit ball in $\R^n.$
\end{rem}
\begin{proof} 
Choose $r>0$ so small that $B_r(x_0)\subset \Omega$.  Note that $u(x_0)-u$ is $p$-harmonic in the punctured ball $B_r(x_0)\setminus\{x_0\}$. Moreover,
\begin{align*}
u(x_0)-u(x)&\le  [u]_{C^{1-n/p}(\R^n)}r^{1-\frac{n}{p}}\\
&=  [u]_{C^{1-n/p}(\R^n)}r^{1-\frac{n}{p}-\left(\frac{p-n}{p-1}\right)}r^{\frac{p-n}{p-1}}\\
&=  [u]_{C^{1-n/p}(\R^n)}r^{1-\frac{n}{p}-\left(\frac{p-n}{p-1}\right)}|x-x_0|^{\frac{p-n}{p-1}}\\
&=: L_r|x-x_0|^{\frac{p-n}{p-1}}
\end{align*}
for $x\in \partial  B_r(x_0)$. Since both $u(x_0)-u(x)$ and $L_r|x-x_0|^{\frac{p-n}{p-1}}$ are $p$-harmonic in $B_r(x_0)\setminus\{x_0\}$ and vanish at $x=x_0$, the maximum principle implies
$$
u(x_0)-u(x)\le L_r|x-x_0|^{\frac{p-n}{p-1}}
$$
for $x\in B_r(x_0)$. With this estimate, the assertion follows from Theorem 1.1 and Remark 1.6 of \cite{KicVer}.
\end{proof}
We have already noted that if $u$ is an extremal whose $1-n/p$ H\"older ratio attains its maximum at two distinct points $x_0,y_0\in \R^n$, then $u$ is continuously differentiable in $\R^n\setminus\{x_0,y_0\}$. It follows from the above lemma that a nonconstant extremal is not differentiable at its maximum and minimum points. 
\begin{cor}
Suppose $u\in {\cal  D}^{1,p}(\R^n)$ is a nonconstant extremal whose $1-n/p$ H\"older ratio attains its maximum at two distinct points $x_0,y_0\in \R^n$.  Then $u$ is not differentiable at 
 $x_0$ or $y_0$. 
\end{cor}

\begin{proof}
When $n=1$, we may conclude the nondifferentiability of $u$ at the points where its $1-1/p$ H\"older ratio attains its maximum by simply inspecting the graph of \eqref{1Dextremalfun}. 

\par Let's now consider the case where $n\ge 2$ and additionally suppose $u(x_0)>u(y_0)$.  As $u$ is an extremal, we have by Proposition \eqref{PDEprop} that for sufficiently small $r>0$
\be
C^p_*\int_{B_r(x_0)}|Du|^{p-2}Du\cdot D\phi dx=\frac{|u(x_0)-u(y_0)|^{p-2}(u(x_0)-u(y_0))}{|x_0-y_0|^{p-n}}\phi(x_0)
\ee
for each $\phi \in C^\infty_c(B_r(x_0))$. That is, 
$$
-\Delta_pu=\frac{|u(x_0)-u(y_0)|^{p-2}(u(x_0)-u(y_0))}{C^p_*|x_0-y_0|^{p-n}}\delta_{x_0}
$$
in $B_r(x_0)$. As $u(x)\le u(x_0)$ in $B_r(x_0)$, we can appeal to Lemma \ref{singularPointLem} to find
$$
n\omega_n\left(\frac{p-n}{p-1}\gamma\right)^{p-1}=\frac{|u(x_0)-u(y_0)|^{p-2}(u(x_0)-u(y_0))}{C^p_*|x_0-y_0|^{p-n}}>0.
$$
In particular, $\gamma>0$. 

\par If $u$ is differentiable at $x_0$,  
$$
u(x)-u(x_0)=Du(x_0)\cdot (x-x_0)+o(|x-x_0|)
$$
as $x\mapsto x_0$.  It would then follow that 
$$
\gamma=\lim_{x\rightarrow x_0}\frac{u(x_0)-u(x)}{|x-x_0|^{\frac{p-n}{p-1}}}\le \lim_{x\rightarrow x_0}\left((|Du(x_0)|+o(1))|x-x_0|^{1-\frac{p-n}{p-1}}\right)=0,
$$
which is a contradiction.   We can also argue similarly for $y_0$.
\end{proof}

%%%%%%%%%%%%%%%%%%%%%%%%%%%%%%%%%%%%%%%%%%%%%%%%%%%%%%%%%%%%%%%%%%%%%%%%%%%
\subsection{Quasiconcavity}
We will now investigate the quasiconcavity and quasiconvexity properties of extremal functions.   Recall that for a given convex $\Omega\subset \R^n$, a function $f:\Omega\rightarrow \R$ is quasiconcave if $\{x\in \Omega: f(x)\ge t\}$ is convex for each $t\in \R$. Likewise, $f$ is quasiconvex if $-f$ is quasiconcave.  We note that $f$ is quasiconcave if 
and only if
\be\label{quasiconcaveIneq}
f((1-\lambda)x+\lambda y)\ge \min\{f(x),f(y)\}
\ee
for $x,y\in \Omega$ and $\lambda\in [0,1]$. 
\begin{prop}\label{QuasiConProp}
Suppose that $u\in {\cal  D}^{1,p}(\R^n)$ is an extremal with 
$$
[u]_{C^{1-n/p}(\R^n)}=\frac{|u(x_0)-u(y_0)|}{|x_0-y_0|^{1-n/p}}
$$
and set 
$$
\Pi_\pm:=\left\{x\in \R^n: \pm \left(x-\frac{1}{2}(x_0+y_0)\right)\cdot(x_0-y_0)>0\right\}.
$$ 
If $u(x_0)>u(y_0)$, then
$$
\mbox{ $u|_{\Pi_+}$ is quasiconcave and  $u|_{\Pi_-}$ is quasiconvex};
$$
alternatively if $u(x_0)<u(y_0)$, then
$$
\mbox{ $u|_{\Pi_+}$ is quasiconvex and  $u|_{\Pi_-}$ is quasiconcave}.
$$
\end{prop}
% 2D figure
\begin{figure}[h]
\centering
 \includegraphics[width=1\columnwidth]{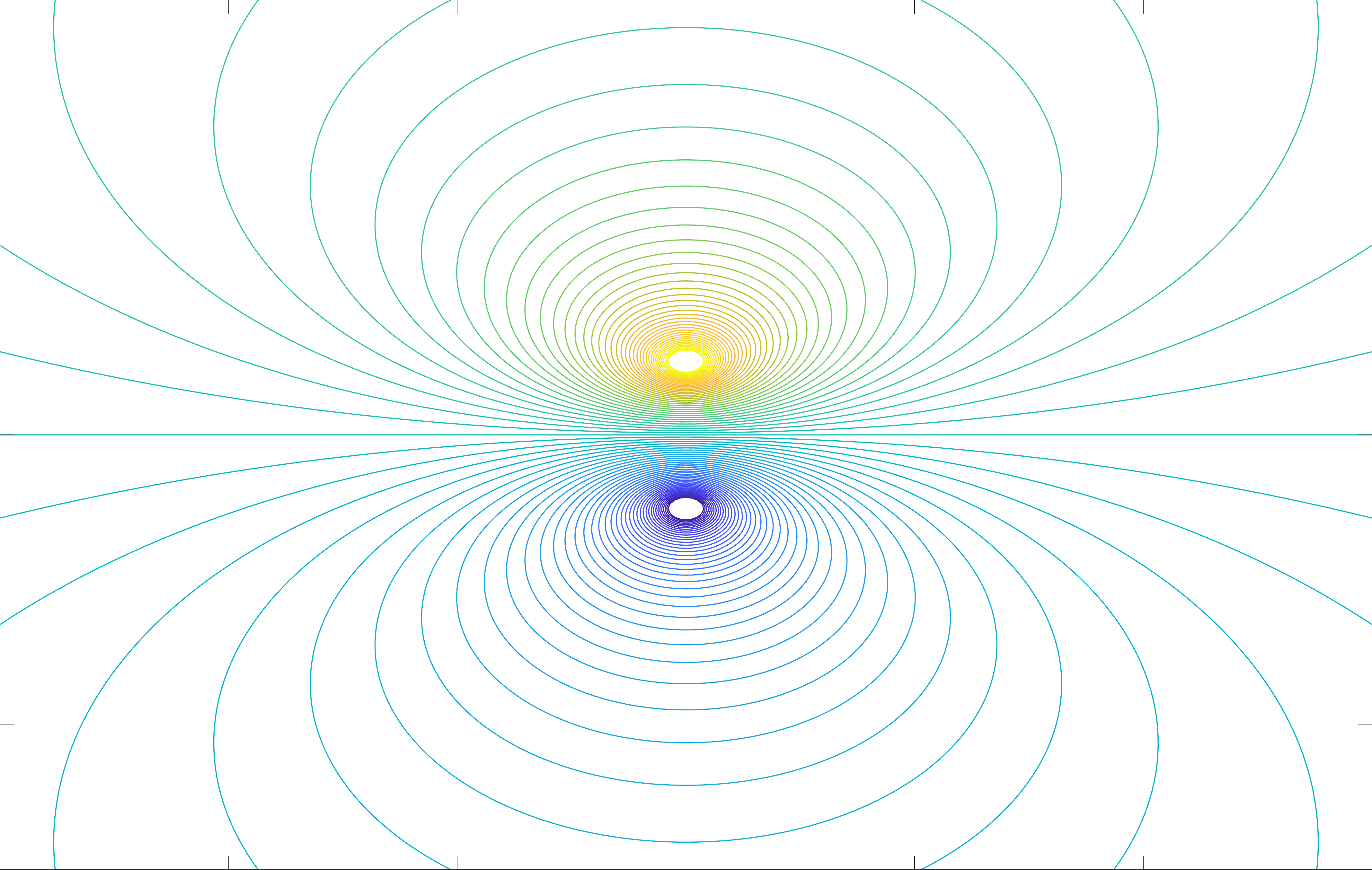}
 \caption{Various contours of a numerical approximation for the extremal $u:\R^2\rightarrow\R$ of Morrey's inequality which satisfies  $u(0,1)=1$ and $u(0,-1)=-1$.  Here  $p=4$.}\label{ContourExt}
\end{figure}
\begin{proof}
It suffices to verify this assertion for the extremal function $u\in {\cal  D}^{1,p}(\R^n)$ which satisfies $u(e_n)=1$ and $u(-e_n)=-1$. In this case, 
$\Pi_\pm=\{x\in \R^n: \pm x_n>0\}$.  By Proposition \ref{reflectionProp} and the fact that the reflection of a convex set about the $x_n=0$ hyperplane is still convex, we only need to verify that $u|_{\Pi_+}$ is quasiconcave. 

\par 1. Recall that $u$ is $p$-harmonic in $\Pi_+\setminus\{e_n\}$, $u(e_n)=1$ and $u|_{\partial \Pi_+}=0$.  We will use this information to build an approximation scheme of quasiconcave functions.  For each $r>0$, set 
$$
D_r:=B_r(\sqrt{1+r^2}e_n)=\{x\in\R^n: x_1^2+\dots + x_{n-1}^2+(x_n-\sqrt{1+r^2})^2<r^2\}.
$$
It's plain that $e_n\in D_r$ and we leave it as an exercise to show that $\overline{D_r}\subset D_s$ for $r<s$ and that $\Pi_+=\bigcup_{r>0}D_r$.  Now let $u_r\in W^{1,p}_0(D_r)$ be the unique solution of the boundary value problem 
\be\label{vrEqn}
\begin{cases}
-\Delta_p v=0&\text{in}\; D_r\\
\hspace{.05in}v(e_n)=1&\\
\hspace{.33in}v=0&\text{on}\; \partial D_r.
\end{cases}
\ee
This function $u_r$ can be found by minimizing the integral 
$$
\int_{D_r}|Dv|^pdx
$$
among $v\in W^{1,p}_0(D_r)$ which satisfy $v(e_n)=1$.  In view of \eqref{vrEqn} and Lemma \ref{singularPointLem}, we also have
\be\label{vrEquation}
-\Delta_pu_r=a_r\delta_{e_n}
\ee
in $D_r$ for some positive constant $a_r$. Multiplying this equation by $u_r$ and integrating by parts shows that in fact $a_r=\int_{D_r}|Du_r|^pdx$.
\par 2. Extending $u_r$ by $0$ to $\Pi_+\setminus D_r$, we have $u_r\in W^{1,p}_0(\Pi_+)$. By the results of Lewis on capacitary functions in convex rings \cite{MR0477094}, $u_r: \Pi_+\rightarrow [0,1]$ is quasiconcave.   In order to complete this proof, we only need to prove that 
\be\label{urconvergenceClaim}
u_r(x)\rightarrow u(x)\quad ( x\in \R^n)
\ee
as $r\rightarrow\infty$. If this is the case, we would have by \eqref{quasiconcaveIneq} that 
$$
u((1-\lambda)x+\lambda y)=\lim_{r\rightarrow\infty}u_r((1-\lambda)x+\lambda y)\ge
\lim_{r\rightarrow\infty}\min\{u_r(x),u_r(y)\}=\min\{u(x),u(y)\}
$$
for each $x,y\in \Pi_+$ and $\lambda\in [0,1]$. So now we focus on proving \eqref{urconvergenceClaim}. 

\par Recall $\overline{D_{r_1}}\subset D_{r_2}$ when $r_1<r_2$, so that $u_{r_1}\in W^{1,p}_0(D_{r_2})$. Moreover,
$$
\int_{\Pi_+}|Du_{r_2}|^pdx=\int_{D_{r_2}}|Du_{r_2}|^pdx
\le \int_{D_{r_2}}|Du_{r_1}|^pdx=\int_{\Pi_+}|Du_{r_1}|^pdx.
$$
Therefore, $(u_r)_{r\ge 1}\subset {\cal  D}^{1,p}(\R^n)$ is bounded in the sense that $u_r(e_n)=1$ and 
$$
\int_{\Pi_+}|Du_{r}|^pdx\le \int_{\Pi_+}|Du_{1}|^pdx
$$
for each $r\ge 1$.  Consequently, there is $u_\infty\in W^{1,p}_0(\Pi_+)$ and a sequence of positive numbers $r_j$ increasing to $\infty$ such that 
$u_{r_j}\rightarrow u_\infty$ locally uniformly in $\Pi_+$ and $Du_{r_j}\rightharpoonup Du_\infty$ in $L^p(\Pi_+;\R^n)$.  In order to verify 
\eqref{urconvergenceClaim}, we only need to show $u_\infty\equiv u|_{\Pi_+}$. 

\par 3.  To this end, we note that since $u_r(e_n)=1$ and $u_r|_{\partial \Pi_+}=0$ for all $r$,  $u_\infty(e_n)=1$ and $u_\infty|_{\partial \Pi_+}=0$. 
And in view of \eqref{vrEquation},
\begin{align*}
 0&\le \int_{D_{r_j}}(|Du_{r_j}|^{p-2}Du_{r_j}-|Dw|^{p-2}Dw)\cdot D(u_{r_j}-w)dx\\
 &=a_{r_j}(u_{r_j}-w)(e_n) -\int_{D_{r_j}}|Dw|^{p-2}Dw\cdot D(u_{r_j}-w)dx\\
 &=a_{r_j}(u_{r_j}-w)(e_n) -\int_{\Pi_+}|Dw|^{p-2}Dw\cdot D(u_{r_j}-w)dx
\end{align*} 
for each fixed $w\in C^\infty_c(\Pi_+)$ and $j\in \N$ sufficiently large.  Sending $j\rightarrow\infty$ gives
\begin{align*}
0&\le a(u_{\infty}-w)(e_n) -\int_{\Pi_+}|Dw|^{p-2}Dw\cdot D(u_{\infty}-w)dx.
\end{align*} 
By approximation, this inequality holds for all $w\in W^{1,p}_0(\Pi_+)$.  

\par 4. For $\lambda>0$, we choose $w=u_{\infty}-\lambda\phi$ for $\phi\in W^{1,p}_0(\Pi_+)$ and note 
$$
0\le a\phi(e_n) -\int_{\Pi_+}|Du_{\infty}-\lambda D\phi|^{p-2}D(u_{\infty}-\lambda \phi)\cdot D\phi dx.
$$
Sending $\lambda\rightarrow 0^+$ gives
$$
0\le a\phi(e_n) -\int_{\Pi_+}|Du_{\infty}|^{p-2}Du_{\infty}\cdot D\phi dx.
$$
Replacing $\phi$ with $-\phi$ gives
$$
\int_{\Pi_+}|Du_{\infty}|^{p-2}Du_{\infty}\cdot D\phi dx=a\phi(e_n) 
$$
for all $\phi\in W^{1,p}_0(\Pi_+)$. In particular, $a=\int_{\Pi_+}|Du_{\infty}|^{p}dx$ so $Du_{r_j}\rightarrow Du_\infty$ in $L^p(\Pi_+;\R^n)$
and $u_\infty$ is $p$-harmonic in $\Pi_+\setminus\{e_n\}$. Since $u_\infty(e_n)=1$ and $u_\infty|_{\Pi_+}=0$, it must be that $u_\infty\equiv u|_{\Pi_+}$. 
\end{proof}
\begin{rem}
It also follows from the proof above that $u_r\rightarrow u$ locally uniformly on $\Pi_+$ and $Du_r\rightarrow Du$ in $L^p(\Pi_+;\R^n)$. 
\end{rem}

%----------------------------------------------------------------------------------------------------------------------------------
\subsection{Analyticity}
It follows from a theorem of Hopf \cite{MR1545250} that a $p$-harmonic function is locally analytic in a neighborhood of any point for which its gradient does not vanish.  Therefore, the analyticity of extremal functions is an immediate corollary of the following assertion. 
\begin{prop}\label{DuNotZeroProp}
Suppose $n\ge 2$ and $u\in {\cal  D}^{1,p}(\R^n)$ is a nonconstant extremal with 
$$
[u]_{C^{1-n/p}(\R^n)}=\frac{|u(x_0)-u(y_0)|}{|x_0-y_0|^{1-n/p}}.
$$
Then  
$$
|Du|>0\quad \text{in}\quad \R^n\setminus\{x_0,y_0\}.
$$
\end{prop}

\begin{proof}
We will assume without loss of generality that $x_0=e_n, y_0=-e_n$ and $u(e_n)=1, u(-e_n)=-1$.  By the antisymmetry of $u$ across $\{x\in \R^n:x\cdot e_n=0\}$ and the fact that $u_{x_n}>0$ on this hyperplane by \eqref{NormalDerivativeu}, it suffices to show that $|Du|>0$ on $\Pi_{+}:=\{x\in \R^n:  x\cdot e_n>0\}$. To this end, we will employ the family of functions $(u_r)_{r>0}$ defined in the proof of Proposition \ref{QuasiConProp}. 

\par In view of \eqref{vrEqn}, $u_r$ is a capacitary function on the convex 
ring $D_r\setminus\{e_n\}$ for each $r>0$.  Lewis showed that for each ball $B$ with $\overline{B}\subset D_r\setminus\{e_n\}$, there is a constant $\rho=\rho(B)\ge 1$ for which
$$
\max_{\overline{B}}|Du_r|\le \rho \min_{\overline{B}}|Du_r|
$$
(section 5 of \cite{MR0477094}). 
 In particular, 
$$
|Du_r(y)|\le \rho |Du_r(z)|
$$
for every $y,z\in \overline{B}$.  

\par Since $Du_{r} \rightarrow Du$ in $L^p(B)$, it must be that $Du_{r_j}(x) \rightarrow Du(x)$ for almost every $x\in  B$ for an appropriate sequence $(r_j)_{j\in \N}$ tending to $0$. Therefore, 
$$
|Du(y)|\le \rho |Du(z)|
$$
for Lebesgue almost every $y,z\in B$. Since $u$ is $p$-harmonic in $B$, $u$ is continuously differentiable in $B$ and this inequality  holds for every $y,z\in B$. That is, 
\be\label{HarnackExtremal}
\max_{\overline{B}}|Du|\le \rho \min_{\overline{B}}|Du|.
\ee

\par   Suppose there is some $y\in \Pi_+$ such that $Du(y)=0.$  Then for every $\delta>0$
for which $B_\delta(y)\subset\Pi_+\setminus\{e_n\}$, $|Du||_{B_\delta(y)}=0$.  This follows from inequality \eqref{HarnackExtremal}.  
Choosing 
$$
\delta:=\text{dist}(y,\partial(\Pi_+\setminus\{e_n\})),
$$
we have that either $e_n\in \partial B_\delta(y)$ or $\partial B_\delta(y)\cap \partial \Pi_+\neq \emptyset$. 
In the case $e_n\in \partial B_\delta(y)$, $u\equiv 1$ in $B_\delta(y)$; otherwise $u\equiv 0$ 
in  $B_\delta(y)$. Either conclusion is a contradiction to Proposition \ref{ubetweenzeroandone}. As a result, no such point $y$ exists.  
\end{proof}

\section{Morrey's estimate and the sharp constant $C_*$}
We will now argue that the sharp constant $C_*$ for Morrey's inequality \eqref{MorreyIneq} cannot be derived the way other constants for Morrey's inequality typically are.  To this end, we will recall Morrey's estimate: there is a constant $C>0$ depending only on $p$ and $n$ such that 
\be\label{babyMorreyIneq}
|u(x)-u(y)|\le C r^{1-n/p}\left(\int_{B_r(z_0)}|Du|^pdz\right)^{1/p}
\ee
for each $x,y\in B_r(z_0)$ and $u\in {\cal  D}^{1,p}(\R^n)$. This estimate was verified by Evans and Gariepy (Theorem 4.10 of \cite{MR3409135}) but is essentially due to Morrey (Lemma 1 of \cite{MR1501936}, Theorem 3.5.2 of \cite{MR2492985}).   We also note that other good presentations of versions of this inequality can be found in Nirenberg's Lecture II of \cite{MR0109940}, section 5.6.2 of Evans' textbook \cite{MR2597943} and in section 7.8 of Gilbarg and Trudinger's monograph \cite{MR1814364}.

\par It is not hard to see that Morrey's inequality holds with the constant $C$ in \eqref{babyMorreyIneq}.  Indeed, we can choose $z_0=x$ and $r=|x-y|$ to find 
\be\label{uLocalHolder}
\frac{|u(x)-u(y)|}{|x-y|^{1-n/p}}\le C\left(\int_{B_r(x)}|Du|^pdz\right)^{1/p}\le C\left(\int_{\R^n}|Du|^pdz\right)^{1/p}
\ee
for each $u\in {\cal  D}^{1,p}(\R^n)$.  However, we claim that any such $C$ must be larger than $C_*$. In particular, it is not possible to find the best constant for 
Morrey's inequality by deriving Morrey's estimate.

\begin{prop}
Suppose $n\ge 2$, $C$ is a constant for which \eqref{babyMorreyIneq} holds, and $C_*$ is the sharp constant for Morrey's inequality. Then 
$$
C_*<C.
$$
\end{prop}
\begin{proof}
Let $u\in{\cal  D}^{1,p}(\R^n)$ be an nonconstant extremal whose H\"older ratio is maximized at $0$ and $e_n$.  
In view of \eqref{uLocalHolder},
\be
C_*\left(\int_{\R^n}|Du|^pdz\right)^{1/p}=\frac{|u(e_n)-u(0)|}{|e_n-0|^{1-n/p}}\le C \left(\int_{B_1(0)}|Du|^pdz\right)^{1/p}.
\ee
By Proposition \ref{DuNotZeroProp}, $|Du|>0$ in $\R^n\setminus B_1(0)$. Therefore, 
$$
C_*\left(\int_{\R^n}|Du|^pdz\right)^{1/p}<C\left(\int_{\R^n}|Du|^pdz\right)^{1/p}.
$$
\end{proof}

%%%%%%%%%%%%%%%%%%%%%%%%%%%%%%%%%%%%%%%%%%%%%%%%%%%%%%%%%%%%%%%%%%%%%%%%%%%
\section{Maximizing the $1-n/p$ H\"older ratio}
We will finally argue that the H\"older ratio of any function belonging to ${\cal  D}^{1,p}(\R^n)$ always attains its maximum at a pair of distinct points.   In particular, every extremal also has this property.  We will also show how the following theorem implies a type of stability for Morrey's inequality. 
\begin{thm}\label{HolderRatioProp}
Assume $u\in {\cal  D}^{1,p}(\R^n)$ is nonconstant. Then 
$$
[u]_{C^{1-n/p}(\R^n)}=\frac{|u(x_0)-u(y_0)|}{|x_0-y_0|^{1-n/p}}
$$
for some $x_0,y_0\in \R^n$ with $x_0\neq y_0$.  
\end{thm}

% Local inequality
\par In order to verify this assertion, it will be convenient for us to employ the following estimate:
\be\label{babyMorreyIneq2}
|u(x)-u(y)|\le C r^{1-n/p}\left(\int_{B_{r/2}\left(\frac{x+y}{2}\right)}|Du|^pdz\right)^{1/p}
\ee
for each $u\in {\cal  D}^{1,p}(\R^n)$ and $x,y\in \R^n$. Here $r=|x-y|$ and $C$ is a constant depending only on $n$ and $p$.
This estimate follows from Morrey's estimate \eqref{babyMorreyIneq} 
by choosing $z_0=\frac{1}{2}(x+y)$ as $x,y\in \partial B_r(z_0)$.  

\par We will prove Theorem \ref{HolderRatioProp} as follows. We first select any pair of sequences $(x_k)_{k\in \N}, (y_k)_{k\in \N}$ such that
$$
[u]_{C^{1-n/p}(\R^n)}=\lim_{k\rightarrow\infty}\frac{|u(x_k)-u(y_k)|}{|x_k-y_k|^{1-n/p}},
$$
and then show 
\be\label{FirstRuleOut}
\liminf_{k\rightarrow\infty}|x_k-y_k|>0
\ee
and  
\be\label{SecondRuleOut}
\sup_{k\in \N}|x_k|,\;\sup_{k\in \N}|y_k|<\infty. 
\ee
It follows that $(x_k)_{k\in \N}$ and $(y_k)_{k\in \N}$ have convergent subsequences $(x_{k_j})_{j\in \N}$ and $(y_{k_j})_{j\in \N}$ that converge to distinct points $x_0$ and $y_0$. We could then conclude by noting
\be
[u]_{C^{1-n/p}(\R^n)}=\lim_{k\rightarrow\infty}\left\{\frac{|u(x_{k_j})-u(y_{k_j})|}{|x_{k_j}-y_{k_j}|^{1-n/p}}\right\}=\frac{|u(x_{0})-u(y_{0})|}{|x_{0}-y_{0}|^{1-n/p}}.
\ee

\par It turns out that \eqref{FirstRuleOut} is easy to prove. For if it fails, we may as well assume $\lim_{k\rightarrow\infty}|x_k-y_k|=0$. Applying \eqref{babyMorreyIneq2} gives  
\begin{align*}
[u]^p_{C^{1-n/p}(\R^n)}&=\limsup_{k\rightarrow\infty}\left\{\frac{|u(x_k)-u(y_k)|^p}{|x_k-y_k|^{p-n}}\right\}\\
&\le C^p\limsup_{k\rightarrow\infty}\int_{B_{r_k/2}\left(\frac{x_k+y_k}{2}\right)}|Du|^pdz \quad\quad (r_k:=|x_k-y_k|)\\
&=0,
\end{align*}
and so $u$ would be constant. The last equality follows from the fact that $|Du|^p$ is an integrable function on $\R^n$ and the Lebesgue measure of $B_{r_k/2}\left(\frac{x_k+y_k}{2}\right)$ tends to $0$ as $k\rightarrow\infty$.   Therefore, we only are left to verify \eqref{SecondRuleOut} in order to conclude Theorem \ref{HolderRatioProp}. We will write separate proofs for $n=1$ and for $n\ge 2$.

%-------------------------------------------------------------------
\subsection{The $1-n/p$ H\"older ratio is maximized for $n=1$}
Suppose $n=1$ and that \eqref{SecondRuleOut} fails. Then there are three possibilities to consider:

\begin{enumerate}[$(i)$]

\item $x_k, y_k\rightarrow \infty$ 

\item $x_k\rightarrow \infty$, $y_k\rightarrow y$

\item $x_k\rightarrow \infty,y_k\rightarrow - \infty $.

\end{enumerate}
We will argue that these three cases cannot occur.

\par \underline{Case $(i)$}:  Fix $\epsilon>0$.  As $|u'|^p$ is integrable, there is $R>0$ such that 
\be\label{IntegralSmall1D}
\left(\int_{\R\setminus[-R,R]}|u'|^pdx\right)^{1/p}\le \epsilon.
\ee
As the $1-1/p$ H\"older ratio of $u$ is a symmetric function, we may assume 
$$
x_k>y_k\ge R
$$ for all $k\in \N$ sufficiently large.  For all 
such $k\in \N$, we can combine \eqref{FunThmCalc} and \eqref{IntegralSmall1D} to find 
\begin{align*}
\frac{|u(x_{k})-u(y_{k})|}{|x_{k}-y_{k}|^{1-1/p}}&\le \left(\int^{x_k}_{y_k}|u'|^pdx\right)^{1/p} \\
&\le \left(\int^{\infty}_{R}|u'|^pdx\right)^{1/p} \\
&\le \epsilon.
\end{align*}
As a result, 
$$
[u]_{C^{1-1/p}(\R)}=\lim_{k\rightarrow\infty}\frac{|u(x_{k})-u(y_{k})|}{|x_{k}-y_{k}|^{1-1/p}}\le  \epsilon
$$
which forces $u$ to be constant.

\par \underline{Case $(ii)$}: We assume $y=0$ and $u(0)=0$. With this assumption, it is routine to check 
\be
[u]_{C^{1-1/p}(\R)}=\lim_{k\rightarrow\infty}\frac{|u(x_{k})-u(0)|}{(x_{k}-0)^{1-1/p}}=\lim_{k\rightarrow\infty}\frac{|u(x_{k})|}{x_{k}^{1-1/p}}.
\ee
We may also suppose that
$$
\frac{|u(x_{k})|}{x_{k}^{1-1/p}}\le \frac{|u(x_{k+1})|}{x_{k+1}^{1-1/p}}
$$
and 
$$
0< 2x_k\le x_{k+1}
$$
for each $k\in \N$. If not, we can pass to appropriate subsequences to achieve these inequalities.  

\par We note
\begin{align*}
\frac{|u(x_{k+1})-u(x_{k})|}{(x_{k+1}-x_{k})^{1-1/p}}&\ge \frac{|u(x_{k+1})|-|u(x_{k})|}{(x_{k+1}-x_{k})^{1-1/p}}\\
&=\frac{x_{k+1}^{1-1/p}(|u(x_{k+1})|/x_{k+1}^{1-1/p})-x_{k}^{1-1/p}(|u(x_{k})|/x_{k}^{1-1/p})}{(x_{k+1}-x_{k})^{1-1/p}}\\
&\ge \frac{|u(x_{k+1})|}{x_{k+1}^{1-1/p}} \frac{x_{k+1}^{1-1/p} -x_{k}^{1-1/p}}{(x_{k+1}-x_{k})^{1-1/p}}\\
&= \frac{|u(x_{k+1})|}{x_{k+1}^{1-1/p}} \frac{1 -(x_{k}/x_{k+1})^{1-1/p}}{(1-(x_{k}/x_{k+1}))^{1-1/p}}\\
&\ge \frac{|u(x_{k+1})|}{x_{k+1}^{1-1/p}}\left(1 -(1/2)^{1-1/p}\right).
\end{align*}
Therefore, 
\be\label{LowerBoundRatio}
\liminf_{k\rightarrow\infty}\frac{|u(x_{k+1})-u(x_{k})|}{|x_{k+1}-x_{k}|^{1-1/p}}\ge [u]_{C^{1-1/p}(\R)}\left(1 -(1/2)^{1-1/p}\right).
\ee
We can now argue as we did for case $(i)$ to find that $u$ is constant.

\par \underline{Case $(iii)$}:   Without loss of generality we may assume $u(0)=0$ and
$$
y_k< 0< x_k
$$
for all $k\in \N$. Observe 
$$
|x_k-y_k|=x_k-y_k>x_k \quad \text{and} \quad |x_k-y_k|=x_k-y_k>-y_k. 
$$
Consequently, 
\begin{align*}
\frac{|u(x_{k})-u(y_{k})|}{|x_{k}-y_{k}|^{1-1/p}}&\le \frac{|u(x_{k})|}{|x_{k}-y_{k}|^{1-1/p}}+\frac{|u(y_{k})|}{|x_{k}-y_{k}|^{1-1/p}}\\
&\le \frac{|u(x_{k})|}{|x_{k}|^{1-1/p}}+\frac{|u(y_{k})|}{|y_{k}|^{1-1/p}}.
\end{align*}
Therefore it must be that either 
$$
\limsup_{k\rightarrow\infty}\frac{|u(x_{k})|}{|x_{k}|^{1-1/p}}>0\quad \text{or} \quad\limsup_{k\rightarrow\infty}\frac{|u(y_{k})|}{|y_{k}|^{1-1/p}}>0.
$$
In either scenario we can argue as in our proof of case $(ii)$ to conclude that $u$ is constant.

%-------------------------------------------------------------------
\subsection{An estimate for the $1-n/p$ H\"older ratio}
Now suppose $n\ge 2$. Observe that if $u\in {\cal  D}^{1,p}(\R^n)$ is nonconstant, there is an $R>0$ such that 
$$
C\left(\int_{\R^n\setminus B_R(0)}|Du|^pdx\right)^{1/p}<\frac{1}{2}[u]_{C^{1-n/p}(\R^n)}.
$$
If in addition 
\be\label{SphereInt}
B_{r/2}\left(\frac{x+y}{2}\right)\subset \R^n\setminus B_R(0)
\ee
with $r=|x-y|>0$, \eqref{babyMorreyIneq2} would give 
$$
\frac{|u(x)-u(y)|}{|x-y|^{1-n/p}}\le C\left(\int_{\R^n\setminus B_R(0)}|Du|^pdx\right)^{1/p}< \frac{1}{2}[u]_{C^{1-n/p}(\R^n)}.
$$
In particular, the supremum of the $1-n/p$ H\"older ratio for $u$ can't be achieved by such a pair $x,y\in \R^n$. 

\par We will argue below that a similar estimate is available for $x,y$ with large norm without requiring  \eqref{SphereInt}.
Our strategy is based on the observation that there are $x, z_1,\dots, z_{m}, y\in \R^n$ with large norm 
such that \eqref{SphereInt} holds for each of the consecutive pairs 
$$
(x,z_1),\dots,(z_i,z_{i+1}),\dots, (z_{m},y).
$$
Moreover, $m$ can be estimated from above provided $|x|,|y|$ are sufficiently large.  In summary, we have the following technical assertion which is proved in the appendix. 
\begin{fcl}
Suppose $R>0$ and $x,y\in \R^n\setminus B_{2R}(0)$. Then there are $m\in\{1,\dots, 8\}$ and $z_1,\dots, z_{m}\in \R^n\setminus B_{2R}(0)$
such that 
\be\label{distancesLessXminusY}
|x-z_1|, \dots,|z_i-z_{i+1}|,\dots,|z_{m}-y|\le  |y-x|
\ee
and 
\be\label{BallsIncludedINBRcomp}
\begin{cases}
B_{r/2}\left(\frac{x+z_1}{2}\right)& \text{with}\quad r=|x-z_1|\\
\hspace{.5in} \vdots\\
B_{r_i/2}\left(\frac{z_i+z_{i+1}}{2}\right) & \text{with}\quad r_i=|z_i-z_{i+1}|\\
\hspace{.5in} \vdots\\
B_{s/2}\left(\frac{z_{m}+y}{2}\right)& \text{with}\quad s=|z_{m}-y|
\end{cases}
\ee
are all subsets of $\R^n\setminus B_R(0)$.
\end{fcl}

% FCL figure
\begin{figure}[h]
\centering
 \includegraphics[width=.8\textwidth]{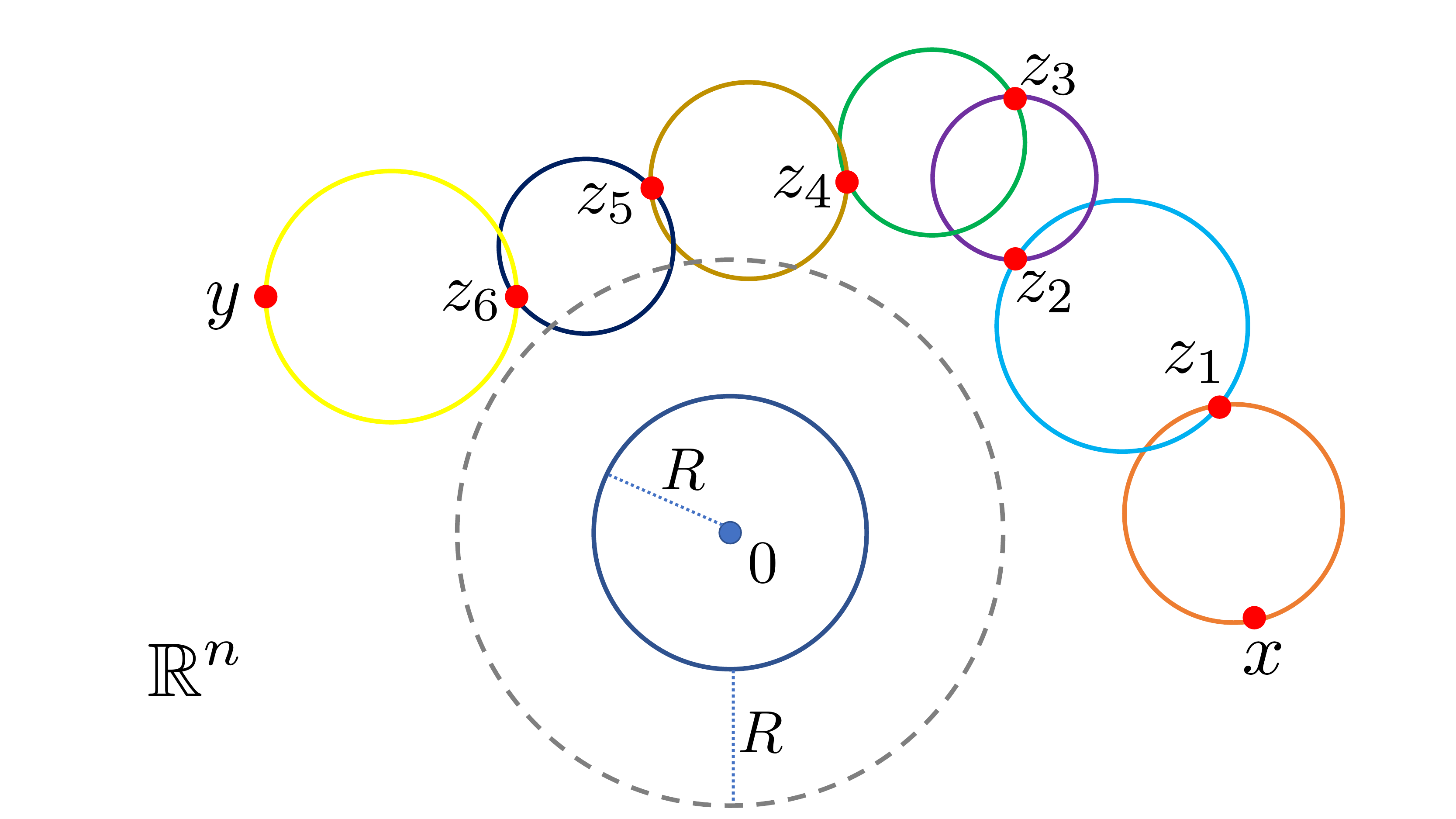}
 \caption{This figure illustrates the conclusion of the Finite Chain Lemma. Note that $x,y\in \R^n\setminus B_{2R}(0)$ and
$z_1,\dots,z_6$ satisfy \eqref{distancesLessXminusY} and \eqref{BallsIncludedINBRcomp}. }%\label{}
\end{figure}

\par The main application of the Finite Chain Lemma is the following assertion. 

 \begin{lem}
Suppose $R>0$, $u\in {\cal  D}^{1,p}(\R^n)$, and $x,y\in \R^n\setminus B_{2R}(0)$.  Then 
\be\label{EstimateForLargeXY}
|u(x)-u(y)|\le 9 C\left(\int_{\R^n\setminus B_R(0)}|Du|^pdx\right)^{1/p}|x-y|^{1-n/p}
\ee
\end{lem}
\begin{proof}
We apply the Finite Chain Lemma to obtain $z_1,\dots, z_{m}\in \R^n\setminus B_{2R}(0)$ with $m\in\{1,\dots, 8\}$ that
satisfy \eqref{distancesLessXminusY} and \eqref{BallsIncludedINBRcomp}.  Next, we employ \eqref{babyMorreyIneq2}  to get 
\begin{align*}
|u(x)-u(y)|&\le |u(x)-u(z_1)|+\sum^{m-1}_{j=1}|u(z_j)-u(z_{j+1})|+|u(z_{m})-u(y)|\\
&=C\left(\int_{\R^n\setminus B_R(0)}|Du|^pdx\right)^{1/p}\left(|x-z_1|^{1-n/p}+\sum^{m-1}_{j=1}|z_j-z_{j+1}|^{1-n/p}+|z_{m}-y|^{1-n/p}\right)\\
&\le (m+1) C\left(\int_{\R^n\setminus B_R(0)}|Du|^pdx\right)^{1/p}|x-y|^{1-n/p}\\
&\le 9 C\left(\int_{\R^n\setminus B_R(0)}|Du|^pdx\right)^{1/p}|x-y|^{1-n/p}.
\end{align*}

\end{proof}

%-------------------------------------------------------------------
\subsection{The $1-n/p$ H\"older ratio is maximized for $n\ge 2$}
We are now in position to verify \eqref{SecondRuleOut} and in turn complete our proof of Theorem \ref{HolderRatioProp} for $n\ge 2$. 
There are two cases to consider:

\begin{enumerate}[$(i)$]

\item $|x_k|,|y_k|\rightarrow \infty$

\item $|x_k|\rightarrow \infty$, $y_k\rightarrow y$.

\end{enumerate}

\par \underline{Case $(i)$}:   
Fix $\epsilon>0$ and choose $R>0$ so large that 
\be\label{9timesDuIntegralSmall}
9C\left(\int_{\R^n\setminus B_R(0)}|Du|^pdx\right)^{1/p}<\epsilon.
\ee
For all large enough $k\in \N$
$$
|y_k|, |x_k|\ge 2R,
$$
so we can combine \eqref{EstimateForLargeXY} and \eqref{9timesDuIntegralSmall} to get
$$
\frac{|u(x_k)-u(y_k)|}{|x_k-y_k|^{1-n/p}}<\epsilon.
$$
Therefore,
$$
[u]_{C^{1-n/p}(\R^n)}=\lim_{k\rightarrow\infty}\frac{|u(x_k)-u(y_k)|}{|x_k-y_k|^{1-n/p}}\le\epsilon.
$$
It follows that $u$ must be constant.  

\par \underline{Case $(ii)$}: It is routine to verify 
$$
[u]_{C^{1-n/p}(\R^n)}=\lim_{k\rightarrow\infty}\frac{|u(x_k)-u(y)|}{|x_k-y|^{1-n/p}}.
$$
We may also assume for each $k\in \N$
$$
\frac{|u(x_k)-u(y)|}{|x_k-y|^{1-n/p}}\le\frac{|u(x_{k+1})-u(y)|}{|x_{k+1}-y|^{1-n/p}}
$$
and
$$
0<2|x_k-y|\le |x_{k+1}-y|,
$$
as these inequalities hold along appropriately selected subsequences.  Observe

\begin{align*}
\frac{|u(x_{k+1})-u(x_k)|}{|x_{k+1}-x_k|^{1-n/p}}&\ge \frac{|u(x_{k+1})-u(y)|-|u(x_k)-u(y)|}{|(x_{k+1}-y)-(x_k-y)|^{1-n/p}}\\
&\ge\frac{|u(x_{k+1})-u(y)|}{|x_{k+1}-y|^{1-n/p}} \frac{|x_{k+1}-y|^{1-n/p}-|x_k-y|^{1-n/p}}{|(x_{k+1}-y)-(x_k-y)|^{1-n/p}}\\
&\ge\frac{ |u(x_{k+1})-u(y)|}{\displaystyle |x_{k+1}-y|^{1-n/p}} \frac{|x_{k+1}-y|^{1-n/p}-\frac{1}{2^{1-n/p}}|x_{k+1}-y|^{1-n/p}}{|(x_{k+1}-y)-(x_k-y)|^{1-n/p}}\\
&=\frac{|u(x_{k+1})-u(y)|}{|x_{k+1}-y|^{1-n/p}}\left(1-\left( \frac{1}{2}\right)^{1-n/p}\right)\frac{|x_{k+1}-y|^{1-n/p}}{|(x_{k+1}-y)-(x_k-y)|^{1-n/p}}.
\end{align*}
\par As
$$
|(x_{k+1}-y)-(x_k-y)|\le |x_{k+1}-y|+|x_k-y|\le \frac{3}{2}|x_{k+1}-y|,
$$
it follows that
$$
\frac{|x_{k+1}-y|^{1-n/p}}{|(x_{k+1}-y)-(x_k-y)|^{1-n/p}}\ge\left( \frac{2}{3}\right)^{1-n/p}.
$$
As a result,
\be\label{xkplusoneandxk}
\liminf_{k\rightarrow\infty}\frac{|u(x_{k+1})-u(x_k)|}{|x_{k+1}-x_k|^{1-n/p}}\ge [u]_{C^{1-n/p}(\R^n)}\left[\left( \frac{2}{3}\right)^{1-n/p}-\left( \frac{1}{3}\right)^{1-n/p}\right].
\ee
We can now argue as in the case $(i)$ to conclude that $u$ must be constant.

%%%%%%%%%%%%%%%%%%%%%%%%%%%%%%%%%%%%%%
\subsection{Stability}
It turns out that a type of stability for Morrey's inequality follows directly from Theorem \ref{HolderRatioProp}. In order to establish this stability, we will make use of Clarkson's inequalities which state: for $f,g\in L^p(\R^n)$
\be\label{FirstClarkson}
\left\|\frac{f+g}{2}\right\|^p_{L^p(\R^n)}+\left\|\frac{f-g}{2}\right\|^p_{L^p(\R^n)}\le\frac{1}{2}\left\|f\right\|^p_{L^p(\R^n)}+\frac{1}{2}\left\|g\right\|^p_{L^p(\R^n)}
\ee
for $2<p<\infty$ and 
$$
\left\|\frac{f+g}{2}\right\|^{\frac{p}{p-1}}_{L^p(\R^n)}+\left\|\frac{f-g}{2}\right\|^{\frac{p}{p-1}}_{L^p(\R^n)}\le\left(\frac{1}{2}\left\|f\right\|^p_{L^p(\R^n)}+\frac{1}{2}\left\|g\right\|^p_{L^p(\R^n)}\right)^{\frac{1}{p-1}}
$$
for $1< p\le 2$. We would also like to emphasize that the method presented in our proof below is quite different from the way stability is typically pursued for Sobolev and related isoperimetric inequalities \cite{MR1124290, MR2538501,MR2395175,MR3896203,MR2456887,MR3404715}.

\begin{cor}
Suppose $v\in {\cal  D}^{1,p}(\R^n)$.  Then there is an extremal $u\in {\cal  D}^{1,p}(\R^n)$ such that
\be\label{StabilitypBig2}
\left(\frac{C_*}{2}\right)^p\|Du-Dv\|^p_{L^p(\R^n)}+[v]^p_{C^{1-n/p}(\R^n)}\le C_*^p\|Dv\|^p_{L^p(\R^n)}
\ee
when $2<p<\infty$ and 
\be
\left(\frac{1}{2}\right)^{\frac{p}{p-1}}\|u'-v'\|^{\frac{p}{p-1}}_{L^p(\R)}+[v]^{\frac{p}{p-1}}_{C^{1-1/p}(\R)}\le \|v'\|^{\frac{p}{p-1}}_{L^p(\R)}
\ee
when $1<p\le 2$.
\end{cor}
\begin{proof}
According to Theorem \ref{HolderRatioProp}, there are distinct $x_0,y_0\in\R^n$ such that 
$$
[v]_{C^{1-n/p}(\R^n)}=\frac{|v(x_0)-v(y_0)|}{|x_0-y_0|^{1-n/p}}.
$$
By Corollary \ref{ExistenceCor}, we can also select an extremal $u$ with $u(x_0)=v(x_0)$, $u(y_0)=v(y_0)$ and 
$$
[u]_{C^{1-n/p}(\R^n)}=[v]_{C^{1-n/p}(\R^n)}.
$$
Of course, we have 
$$
\|Du\|_{L^p(\R^n)}\le \|Dv\|_{L^p(\R^n)}.
$$
It is also easy to check that 
$$
[v]_{C^{1-n/p}(\R^n)}=\left[\frac{u+v}{2}\right]_{C^{1-n/p}(\R^n)}.
$$

\par First suppose $2<p<\infty$. We note that inequality \eqref{FirstClarkson}, while stated for functions, also holds for measurable mappings of $f,g:\R^n\rightarrow\R^n$. Therefore, 
\begin{align*}
\left(\frac{C_*}{2}\right)^p\left\|Du-Dv\right\|^p_{L^p(\R^n)}+[v]_{C^{1-n/p}(\R^n)}^p&=C^p_*\left\|\frac{Du-Dv}{2}\right\|^p_{L^p(\R^n)}+[v]_{C^{1-n/p}(\R^n)}^p\\
&=C^p_* \left\|\frac{Du-Dv}{2}\right\|^p_{L^p(\R^n)}+\left[\frac{u+v}{2}\right]^p_{C^{1-n/p}(\R^n)}\\
&\le  C^p_*\left\|\frac{Du-Dv}{2}\right\|^p_{L^p(\R^n)}+ C^p_*\left\|\frac{Du+Dv}{2}\right\|^p_{L^p(\R^n)}\\
&=  C^p_*\left(\left\|\frac{Du-Dv}{2}\right\|^p_{L^p(\R^n)}+ \left\|\frac{Du+Dv}{2}\right\|^p_{L^p(\R^n)}\right)\\
&\le  C^p_*\left(\frac{1}{2}\|Du\|^p_{L^p(\R^n)}+ \frac{1}{2}\|Dv\|^p_{L^p(\R^n)}\right)\\
&\le  C^p_*\|Dv\|^p_{L^p(\R^n)}.
\end{align*}

\par Recall that when $1<p\le 2$, $n$ is necessarily equal to $1$. As a result, we can apply the other Clarkson's inequality to find
\begin{align*}
\left(\frac{1}{2}\right)^{\frac{p}{p-1}}\|u'-v'\|^{\frac{p}{p-1}}_{L^p(\R)}+[v]_{C^{1-1/p}(\R)}^{\frac{p}{p-1}} &=\left\|\frac{u'-v'}{2}\right\|^{\frac{p}{p-1}}_{L^p(\R)}+[v]_{C^{1-1/p}(\R)}^{\frac{p}{p-1}} \\
&=\left\|\frac{u'-v'}{2}\right\|^{\frac{p}{p-1}}_{L^p(\R)}+\left[\frac{u+v}{2}\right]_{C^{1-1/p}(\R)}^{\frac{p}{p-1}}\\
&\le \left\|\frac{u'-v'}{2}\right\|^{\frac{p}{p-1}}_{L^p(\R)}+ \left\|\frac{u'+v'}{2}\right\|^{\frac{p}{p-1}}_{L^p(\R)}\\
&\le \left(\frac{1}{2}\left\|u'\right\|^p_{L^p(\R)}+\frac{1}{2}\left\|v'\right\|^p_{L^p(\R)}\right)^{\frac{1}{p-1}}\\
&\le  \|v'\|^{\frac{p}{p-1}}_{L^p(\R)}.
\end{align*}
\end{proof}
%\begin{rem}
%A similar method has been used recently by Neumeyer to obtain stability estimates for Sobolev's inequality \cite{MR4048334}. 
%\end{rem}

% Finite Chain Lemma
\appendix

\section{Finite Chain Lemma}\label{FCL}
We will first pursue the Finite Chain Lemma for $n=2$. To this end, we will need to recall a few facts about isosceles triangles. Suppose $a>0$ and consider an isosceles triangle with two sides equal to $1+a$ and another equal to $a$. Using the law of cosines we find that the angle $\theta(a)$ between the two sides of length $1+a$ satisfies 
\be\label{thetaA}
\cos(\theta(a))=1-\frac{1}{2}\left(\frac{a}{1+a}\right)^2.
\ee
It is also easy to check that $0<\theta(a)<\pi/3$ and that $\theta(a)$ is increasing in $a$. In particular, $\theta(1)\le \theta(a)$ for $a\ge1$. 
It will also be useful to note
$$
\theta(1)=\cos^{-1}\left(\frac{7}{8}\right) >\frac{\pi}{7}.
$$

\par With these observations, we can derive the following assertion.

\begin{lem}\label{BabyGeometryLemma}
Assume $a\ge 1$ and $x,y\in \R^2$ with 
$$
|y|= |x|=1+a.
$$
If 
$$
|y-x|>a,
$$
there are $z_1,\dots, z_m\in \R^2$ with $m\in \{1,\dots, 7\}$ such that 
$$
|z_1|=\dots=|z_m|=1+a,
$$
\be\label{SequenceDifferences}
|x-z_1|=|z_1-z_2|=\dots=|z_{m-1}-z_m|=a,
\ee
and
$$
\left|y-z_m\right|\le a. 
$$
\end{lem}

% Lemma figure
\begin{figure}[h]
\centering
 \includegraphics[width=.8\textwidth]{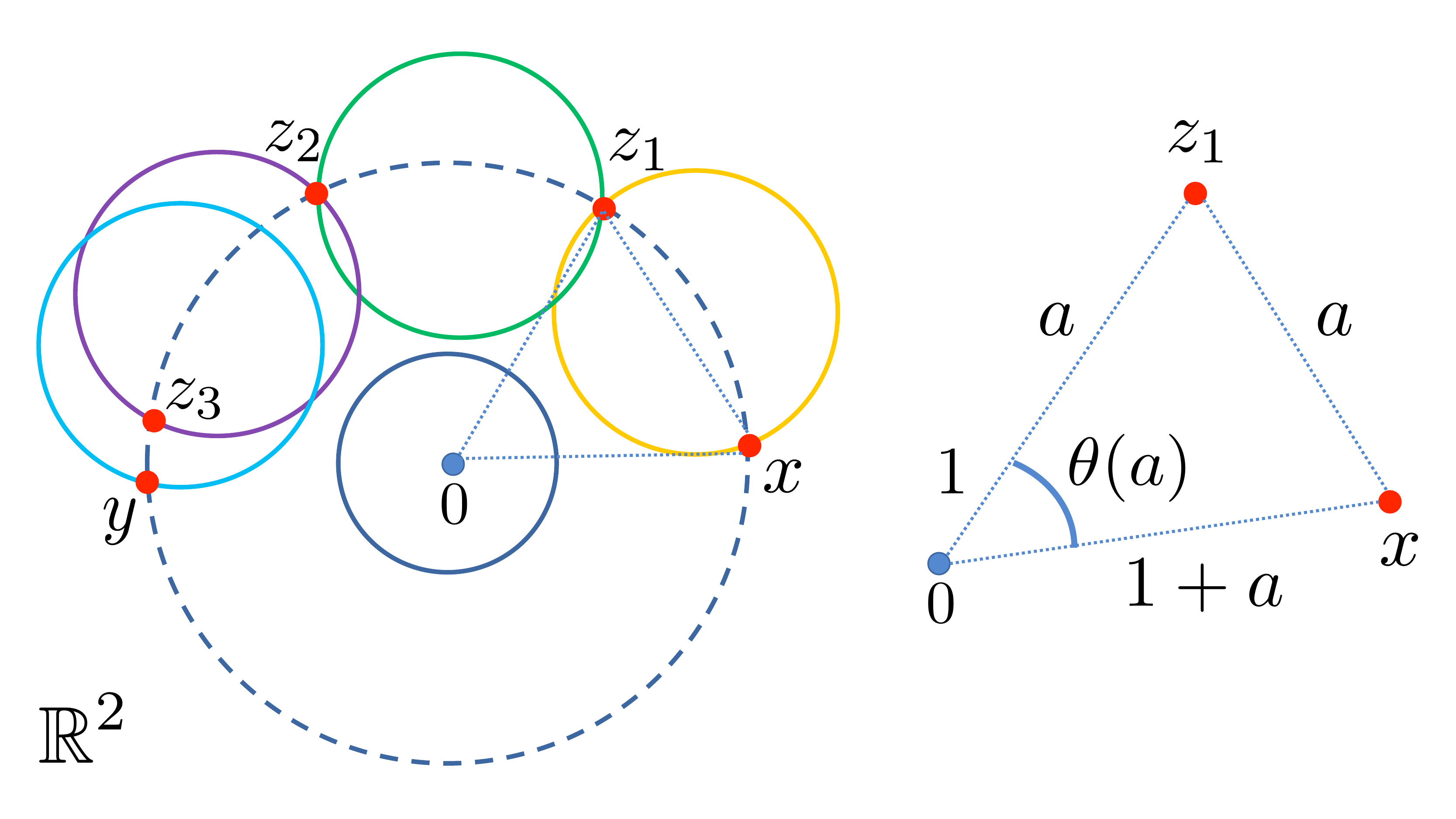}
 \caption{This figure on the left is a schematic of the assertion made in Lemma \ref{BabyGeometryLemma}. This corresponds to the case where 
 three points $z_1,z_2,z_3$ are needed to form a finite chain linking $x$ to $y$ on the circle centered at $0$ of radius $1+a$. The enlarged triangle 
 on the right shows how $\theta(a)$ is defined and how it is related to the figure on the left.}%\label{}
\end{figure}

\begin{proof}
First assume $x=(1+a)e_1$.  Note that we can write 
$$
y=(1+a)(\cos\vartheta,\sin\vartheta).
$$
for some $\vartheta\in (0,2\pi)$. If $\vartheta\in (0,\pi]$, we set 
$$
z_j=(1+a)(\cos(j\theta(a)),\sin(j\theta(a)))
$$
$j=1,\dots, 7$.   By the definition of $\theta(a)$, \eqref{SequenceDifferences} holds.  

\par As $7\theta(a)>\pi$, there is a $m\in \{1,\dots, 7\}$ such that 
$$
(m-1)\theta(a)\le \vartheta< m\theta(a). 
$$
It follows that
\begin{align*}
|y-z_m|^2&=(1+a)^22\left[1-\cos(m\theta(a)-\vartheta)\right]\\
&\le (1+a)^22\left[1-\cos(\theta(a))\right]\\
&=a^2.
\end{align*}

\par If $\vartheta\in (\pi,2\pi)$ we can reason as in the case when $\vartheta\in (0,\pi]$ for $-y$ and obtain points $z_1,\dots, z_m\in \R^2$. In this case, $-z_1,\dots, -z_m$ satisfy the conclusion of this lemma.  For a general $x$ that is not necessarily equal to  $(1+a)e_1$, we can find a rotation $O$ of $\R^2$ so that $Ox=(1+a)e_1$. Then we can prove the assertion for $(1+a)e_1$ and $Oy$ to get points $z_1,\dots, z_m\in \R^2$ satisfying the conclusion of the lemma as we argued above. Then $O^{-1}z_1,\dots, O^{-1}z_m$ satisfy the conclusion of this lemma for the given $x$ and $y$. 
\end{proof}

\begin{cor}\label{xandyequalCor}
Assume $s\ge t>0$ and $x,y\in \R^2$ with 
$$
|y|= |x|=t+s.
$$
If 
$$
|y-x|>s,
$$
there are $z_1,\dots, z_m\in \R^2$ with $m\in \{1,\dots, 7\}$ such that 
\be\label{zmNorms}
|z_1|=\dots=|z_m|=t+s,
\ee
\be\label{SequenceDifferencesAgain}
|x-z_1|=|z_1-z_2|=\dots=|z_{m-1}-z_m|=s,
\ee
and
\be\label{diffYandZm}
\left|y-z_m\right|\le s. 
\ee
\end{cor}
\begin{proof}
We can apply the previous lemma with $a=s/t\ge1$ and $x/t, y/t\in \R^2$.  
\end{proof}
We are of course interested in the scenario where $|x|$ is not necessarily equal to $|y|$. Fortunately, we can just add an additional point to obtain an analogous statement.  We also will need to use the following elementary fact: if $x,y\in \R^2$ with $|y|\ge |x|>0$, then 
\be\label{ElemdistIneq}
\left||x|\frac{y}{|y|}-x\right|\le |x-y|.
\ee

\begin{cor}\label{ChainCor}
Assume $s\ge t>0$ and $x,y\in \R^2$ with 
$$
|y|\ge |x|=t+s.
$$
Suppose
$$
\left||x|\frac{y}{|y|}-x\right|>s.
$$
(i) Then there are $z_1,\dots, z_{m}\in \R^2$ with $m\in \{1,\dots, 8\}$ such that 
$$
|z_1|=\dots=|z_{m}|=t+s,
$$
and
\be\label{SequenceDifferencesOnceAgain}
|x-z_1|, \dots,|z_i-z_{i+1}|,\dots,|z_{m}-y|\le  |y-x|.
\ee
(ii) Furthermore,
\be\label{ChainSubsets}
\begin{cases}
B_{t_0/2}\left(\frac{x+z_1}{2}\right)&\text{with}\quad t_0=|x-z_1|\\
\hspace{.5in} \vdots\\
B_{t_i/2}\left(\frac{z_i+z_{i+1}}{2}\right)&\text{with}\quad t_i=|z_i-z_{i+1}|\\
\hspace{.5in} \vdots\\
B_{t_{m}/2}\left(\frac{z_{m}+y}{2}\right)&\text{with}\quad t_{m}=|z_{m}-y|
\end{cases}
\ee
are all subsets of $\R^2\setminus B_t(0)$.
\end{cor}
\begin{proof}
$(i)$ Corollary \ref{xandyequalCor} applied to $x$ and $|x|\frac{y}{|y|}$ give at most seven points $z_1,\dots, z_{m-1}$ such 
that \eqref{zmNorms} holds.   We also observe that by \eqref{SequenceDifferencesAgain} and inequality \eqref{ElemdistIneq} 
$$
|x-z_1|=|z_1-z_2|=\dots=|z_{m-2}-z_{m-1}|=s<\left||x|\frac{y}{|y|}-x\right|\le |y-x|.
$$
Set
$$
z_{m}:=|x|\frac{y}{|y|},
$$
and note by conclusion \eqref{diffYandZm} of the previous corollary  that $|z_{m-1}-z_m|\le s\le |y-x|$. 
Moreover, $|z_{m}|=|x|=t+s$.  As
$$
|z_{m}-y|=\left||x|\frac{y}{|y|}-y\right|=|y|-|x|\le |y-x|,
$$
we have verified each inequality in \eqref{SequenceDifferencesOnceAgain}.

\par $(ii)$ Since $|x|=s+t$, $B_s(x)\subset\R^2\setminus B_t(0)$. Moreover, if $z\in B_{s/2}\left(\frac{x+z_1}{2}\right)$ then 
$$
|z-x|\le \left|z- \frac{x+z_1}{2}\right| + \left|\frac{x+z_1}{2} -z_1\right|\le \frac{s}{2}+\frac{|x-z_1|}{2}=s.
$$
Thus, $B_{s/2}\left(\frac{x+z_1}{2}\right)\subset B_s(x)\subset \R^2\setminus B_t(0)$ with $s=|x-z_1|$. The inclusions for $i=1,\dots, m-1$ in \eqref{ChainSubsets} follow similarly.  

\par As for the case $i=m$,  set $t_m:=|z_{m}-x|=|y|-|x|$.   Note that the closest point in $B_{t_m/2}\left(\frac{z_{m}+y}{2}\right)$ to the origin is $z_{m}=|x|\frac{y}{|y|}$. Thus for any $z\in B_{t_m/2}\left(\frac{z_{m}+y}{2}\right)$, $|z|\ge ||x|\frac{y}{|y|}|=|x|>t$. Hence, $B_{t_m/2}\left(\frac{z_{m}+y}{2}\right)\subset \R^2\setminus B_t(0)$ and we conclude \eqref{ChainSubsets}.
\end{proof}

% Cor figure
\begin{figure}[h]
\centering
 \includegraphics[width=.8\textwidth]{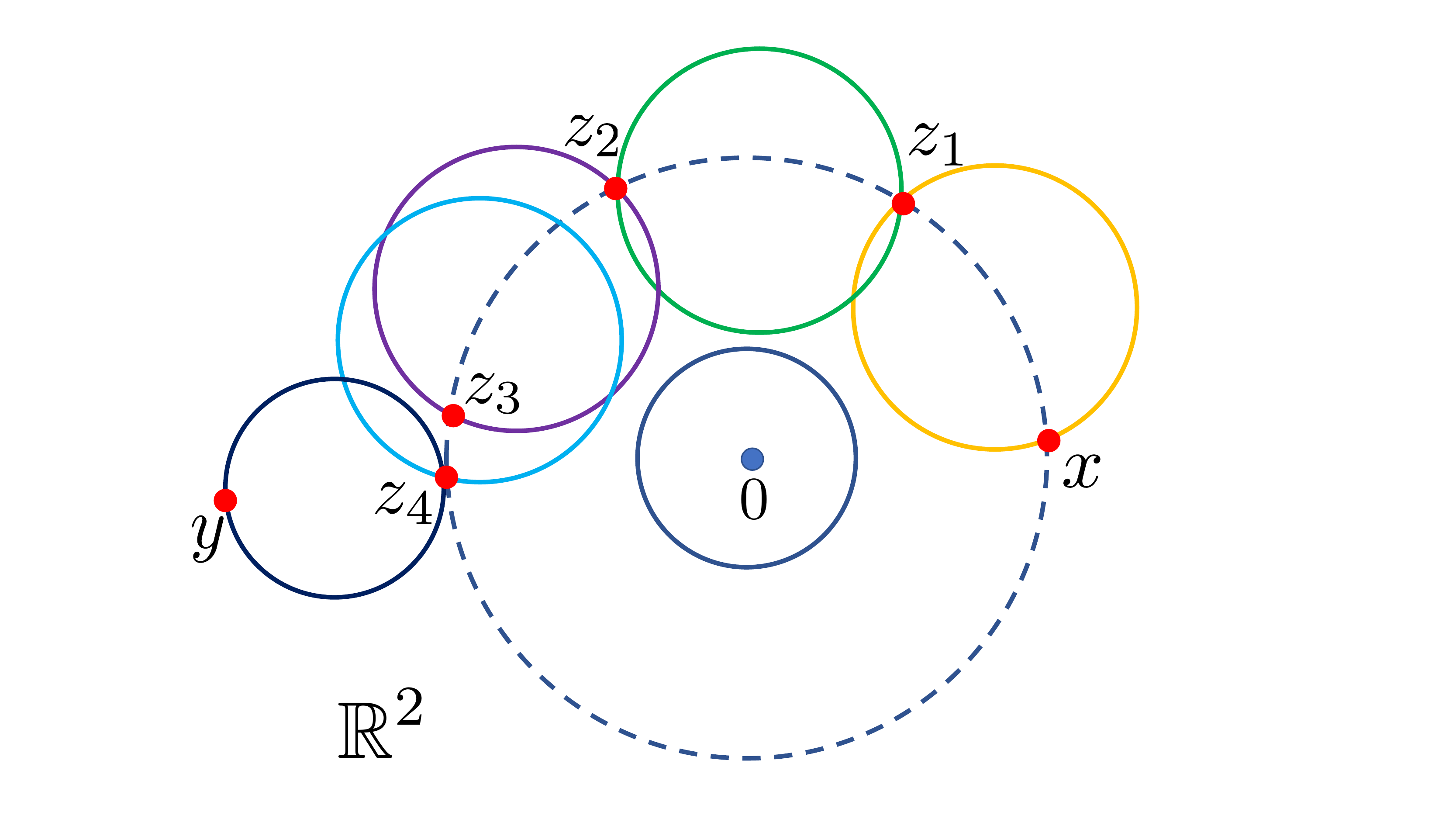}
 \caption{This diagram show how we can adapt our proof of Lemma \ref{BabyGeometryLemma} to the case where $|y|>|x|$. We do so by simply adding another point $z_4=(|x|/|y|)y$ to the chain we obtained by linking $x$ and $(|x|/|y|)y$.}%\label{}
\end{figure}

\par It turns out that we can easily generalize the ideas we developed for $n=2$ to all $n\ge 3$.  The main insight is that if $x$ and $y$ are linearly independent in $\R^n$, they generate  a two dimensional subspace and we can find a chain of points linking $x$ to $y$ that belong to this subspace. In particular, we can accomplish this task by applying Corollary \ref{ChainCor}.

\begin{cor}\label{ChainProp} Suppose $n\ge 2$.  Assume $s\ge t>0$ and $x,y\in \R^n$ with 
$$
|y|\ge |x|=t+s.
$$
Suppose
\be\label{xandymustbelinearlyindep}
\left||x|\frac{y}{|y|}-x\right|>s.
\ee
(i) Then there are $z_1,\dots, z_{m}\in \R^n$ with $m\in \{1,\dots, 8\}$ such that 
$$
|z_1|=\dots=|z_{m}|=t+s,
$$
and
\be%\label{SequenceDifferencesOnceAgain}
|x-z_1|, \dots,|z_i-z_{i+1}|,\dots,|z_{m}-y|\le  |y-x|.
\ee
(ii) Furthermore,
\be%\label{ChainSubsets}
\begin{cases}
B_{t_0/2}\left(\frac{x+z_1}{2}\right)&\text{with}\quad t_0=|x-z_1|\\
\hspace{.5in} \vdots\\
B_{t_i/2}\left(\frac{z_i+z_{i+1}}{2}\right)&\text{with}\quad t_i=|z_i-z_{i+1}|\\
\hspace{.5in} \vdots\\
B_{t_{m}/2}\left(\frac{z_{m}+y}{2}\right)&\text{with}\quad t_{m}=|z_{m}-y|
\end{cases}
\ee
are all subsets of $\R^n\setminus B_t(0)$.\end{cor}
\begin{proof}
The assumption \eqref{xandymustbelinearlyindep} implies that $x$ and $y$ are linearly independent.  In particular, these vectors span the two dimensional subspace 
$$
\{\alpha x+\beta y\in \R^n: \alpha,\beta\in \R\}
$$
of $\R^n$.  We can then apply Corollary \ref{ChainCor} to obtain 
$z_1,\dots, z_{m+1}\in \R^n$ that belong to this subspace and check that these points satisfy the desired conclusions. We leave the details to the reader.  
\end{proof}

\begin{proof}[Proof of the Finite Chain Lemma]
Without loss of generality, we assume $|y|\ge |x|$.  Set $S:=|x|-R$, and note $S\ge R$ with
$$
|y|\ge |x|=R+S.
$$
We will verify the claim by considering the following two cases.

\begin{enumerate}[$(i)$]

\item $\left||x|\frac{y}{|y|}-x\right|\le S$

\item $\left||x|\frac{y}{|y|}-x\right|> S$

\end{enumerate}

\par \underline{Case $(i)$}:  As $|x|=R+S$, $B_S(x)\subset \R^n\setminus B_R(0)$.  And it is straightforward to verify  
$$
B_{r/2}\left(\frac{x+(|x|/|y|)y}{2}\right)\subset B_S(x)%\subset \R^n\setminus B_R(0)
$$
for $r=\left||x|\frac{y}{|y|}-x\right|$ as in the proof of Corollary \ref{ChainCor}.  Likewise we can check that
$$
B_{s/2}\left(\frac{y+(|x|/|y|)y}{2}\right)\subset \R^n\setminus B_R(0)
$$
for $s=\left||x|\frac{y}{|y|}-y\right|$.  We also have $r\le |x-y|$ by inequality \ref{ElemdistIneq} and $s\le |y-x|$ by the triangle inequality.  Therefore, the claim holds for $m=1$ and
$$
z_1=|x|\frac{y}{|y|}.
$$
\par \underline{Case $(ii)$}:  We can apply Corollary \ref{ChainProp} with $s=S$ and $t=R$ to obtain an $m=\{1,\dots,8\}$ and 
a finite sequence $z_1,\dots, z_m\in\R^n\setminus B_{2R}(0)$ which satisfy \eqref{distancesLessXminusY} and \eqref{BallsIncludedINBRcomp}.
\end{proof}

% Gradient descent
\section{Coordinate gradient descent}\label{AppNum}
In this section, we will change notation and use $(x,y)$ to denote a point in $\R^2$. For a given $\ell>1$, we seek to approximate minimizers of the two dimensional integral 
$$
\int^\ell_{-\ell}\int^{\ell}_{-\ell}|Dv(x,y)|^pdxdy
$$
among functions $v\in W^{1,p}([-\ell,\ell]^2)$ which satisfy  
\be\label{2DConstraint}
v(0,1)= 1\quad\text{and}\quad v(0,-1)=-1. 
\ee
It can be shown that a unique minimizer $u_\ell\in W^{1,p}([-\ell,\ell]^2)$ exists and  that $u_\ell$ is $p$-harmonic and thus continuously differentiable in $(-\ell,\ell)^2\setminus\{(0,\pm1)\}$. Moreover, as $\ell\rightarrow\infty$, $u_\ell$ converges locally uniformly to an extremal of Morrey's inequality in $\R^2$ which satisfies \eqref{2DConstraint}. Therefore, our numerical approximation for $u_\ell$ will in turn serve as an approximation for the corresponding extremal. 

\par To this end, we suppose that $\ell\in \N$ and divide the interval $[-\ell,\ell]$ into $N-1$ evenly spaced sub-intervals of length 
$$
h=\frac{2\ell}{N-1}.
$$
Along the $x$-axis, we will label the endpoints of these intervals with
$$
x_i=-\ell+(i-1)h
$$
for $i=1,\dots, N$ and along the $y$-axis we will use the labels
$$
y_j=-\ell+(j-1)h
$$
for $j=1,\dots, N$.  

\par Assuming that $v:[-\ell,\ell]^2\rightarrow \R$ is continuously differentiable 
\begin{align*}
\int^\ell_{-\ell}\int^{\ell}_{-\ell}|Dv(x,y)|^pdxdy&\approx\sum^{N-1}_{i,j=1}|Dv(x_i,y_j)|^ph^2\\
&=\sum^{N-1}_{i,j=1}\left(v_{x}(x_{i},y_{j})^2+v_{y}(x_{i},y_{j})^2\right)^{p/2}h^2\\
&\approx\sum^{N-1}_{i,j=1}\left( \left(\frac{v(x_{i}+h,y_{j})-v(x_{i},y_{j})}{h}\right)^2+\left(\frac{v(x_i,y_j+h)-v(x_i,y_i)}{h}\right)^2\right)^{p/2}h^2\\
&=\sum^{N-1}_{i,j=1}\left( \left(\frac{v(x_{i+1},y_j)-v(x_i,y_i)}{h}\right)^2+\left(\frac{v(x_i,y_{j+1})-v(x_i,y_i)}{h}\right)^2\right)^{p/2}h^2\\
&= h^{2-p}\sum^{N-1}_{i,j=1}\left( \left(v(x_{i+1},y_j)-v(x_i,y_i)\right)^2+\left(v(x_i,y_{j+1})-v(x_i,y_i\right)^2\right)^{p/2}\\
&=h^{2-p}\sum^{N-1}_{i,j=1}\left( \left(v_{i+1,j}-v_{i,j}\right)^2+\left(v_{i,j+1}-v_{i,j}\right)^2\right)^{p/2}.
\end{align*} 
Here we have written
$$
v_{i,j}=v(x_i,y_j).
$$

\par We now suppose that $N$ is of the form
$$
N=2\ell k+1
$$
for some $k\in \N$; this assumption is equivalent to $h=1/k$. We can then attempt to minimize 
\be\label{DiscreteEnergy}
E(v):=\sum^{N-1}_{i,j=1}\left( \left(v_{i+1,j}-v_{i,j}\right)^2+\left(v_{i,j+1}-v_{i,j}\right)^2\right)^{p/2}
\ee
among the $N^2-1$ variables 
$$
v=
\left(\begin{array}{ccccc}
v_{1,1} & v_{1,2}&\dots& v_{1,N-1}& v_{1,N}\\
v_{2,1} & v_{2,2}&\dots& v_{2,N-1}& v_{2,N}\\
\vdots & \vdots &  \ddots & \vdots & \vdots \\
v_{N-1,1} & v_{N-1,2} & \dots & v_{N-1,N-1} &v_{N-1,N} \\
v_{N,1} & v_{N,2} & \dots & v_{N,N-1} &
\end{array}\right),
$$
which satisfy 
\be\label{DiscreteConstraints}
v_{\ell k+1, (\ell+1)k+1}=1\quad \text{and}\quad  v_{\ell k+1, (\ell-1)k+1}=-1.
\ee
These constraints are natural as $x_{\ell k+1}=0, y_{(\ell +1)k+1}=1$ and $y_{(\ell -1)k+1}=-1$. 

\par We are now in position to use coordinate gradient descent to minimize $E$. That is, we choose an initial guess $v^0=(v^0_{i,j})$ which satisfies \eqref{DiscreteConstraints}, select a small
parameter $\tau>0$, and then run the iteration scheme 
$$
\begin{cases}
\displaystyle v^{m}_{i,j}=v^{m-1}_{i,j}-\tau \frac{\partial E(v^{m-1})}{\partial v_{i,j}},\quad &(i,j)\neq (\ell k+1, (\ell\pm1)k+1)\\\\
\displaystyle v^{m}_{i,j}=v^{m-1}_{i,j},\quad &(i,j)=(\ell k+1, (\ell\pm1)k+1)
\end{cases}
$$
for $m\in \N$. After we perform this scheme for large number of iterates $m=1,\dots, M$, for $\tau$ small and $k$ sufficiently large, we can use $v^M_{i,j}$ as an approximation for $u_\ell(x_i,y_j)$. This is what we did to produce the graphs in Figures \ref{ExtremalFig} and \ref{ExtremalFig2} and the contour plot in Figure \ref{ContourExt}.  
% 2D figure
\begin{figure}[h]
\centering
 \includegraphics[width=1\columnwidth]{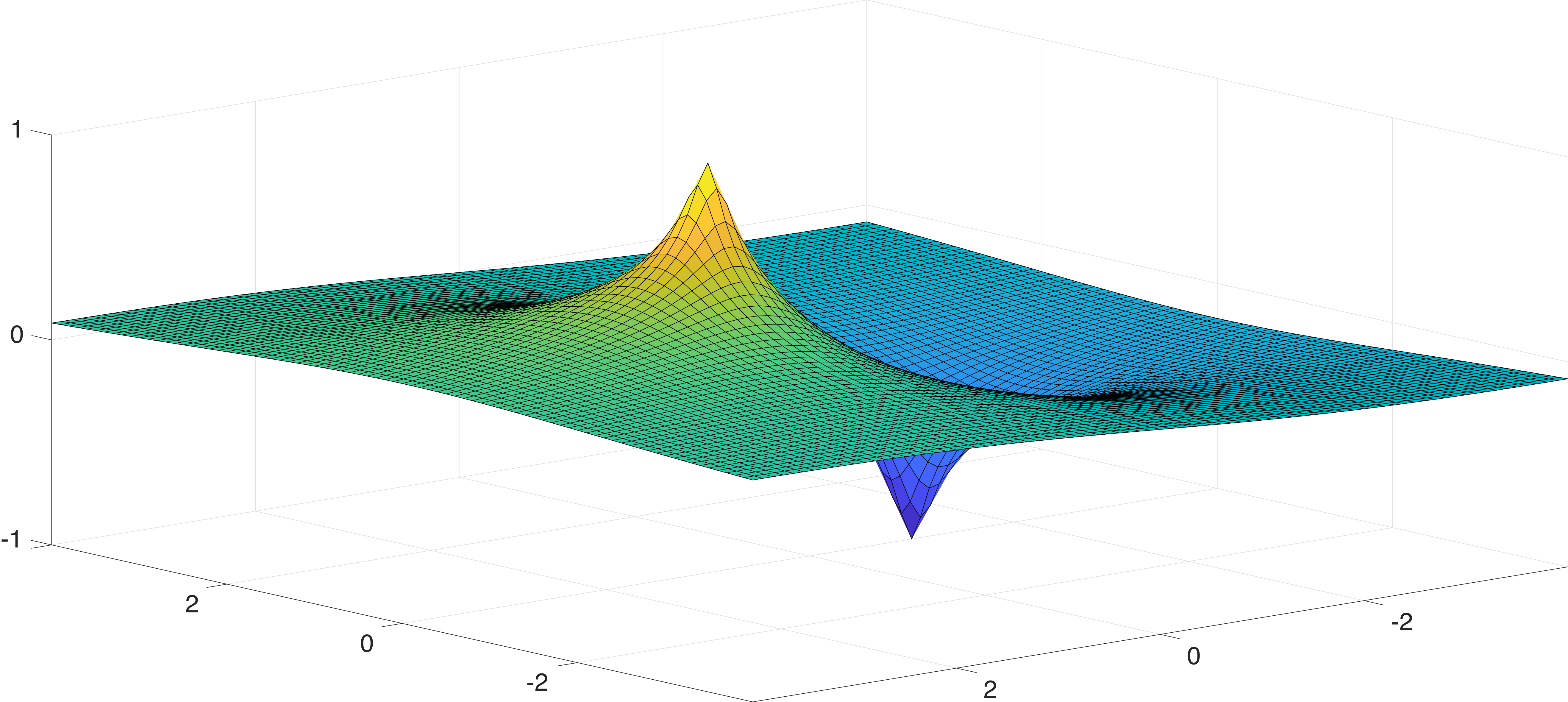}
 \caption{A numerically computed approximation for an extremal of Morrey's inequality with $n=2$ and $p=4$. Here $\ell=6$ and $k=10$ (so that $N=121$), $\tau=10^{-10}$, and this  approximation was obtained after $10^8$  iterations. Our initial guess was $v^0_{i,j}=w(x_i,y_j)$,
 where $w(x,y)=c\ln[(x^2+(y-1)^2+10^{-2})/(x^2+(y+1)^2+10^{-2})]$ and $c$ is chosen to ensure $w(0,1)=1$ and $w(0,-1)=-1$.}
 \label{ExtremalFig2}
\end{figure}

\bibliography{MorreyBib}{}

\bibliographystyle{plain}

\typeout{get arXiv to do 4 passes: Label(s) may have changed. Rerun}

\end{document}